\documentclass[11pt]{article}

\usepackage{hyperref}
\usepackage[normalem]{ulem}
\usepackage{amsthm}
\usepackage{amsmath}
\usepackage{amssymb}
\usepackage{bm}
\usepackage{mathrsfs}
\usepackage{amsfonts}
\usepackage{mathtools}
\usepackage{graphicx}
\usepackage{tikz}
\usetikzlibrary{arrows}
\usepackage{enumitem}
\usepackage{comment}
\usepackage{caption}
\usepackage{titlesec}
\usepackage{upgreek}
\usepackage{fullpage}
\usepackage{appendix}

\titleformat{\section}[hang]{\normalfont\bfseries}{\thesection.}{.5em}{}[]
\titlespacing{\section}{\parindent}{3ex plus .1ex minus .2ex}{1em}

\titleformat{\subsection}[runin]{\normalfont\itshape}{\thesubsection.}{.5em}{}[.]
\titlespacing{\subsection}{\parindent}{3ex plus .1ex minus .2ex}{1em}

\newtheorem{theorem}{Theorem}

\newtheorem{lemma}[theorem]{Lemma}
\newtheorem{proposition}[theorem]{Proposition}
\newtheorem{corollary}[theorem]{Corollary}

\theoremstyle{definition}
 
\theoremstyle{remark}

\theoremstyle{remark}

\numberwithin{theorem}{section}

\providecommand{\R}{}
\providecommand{\Z}{}

\providecommand{\C}{}

\renewcommand{\R}{\mathbb{R}}
\renewcommand{\Z}{\mathbb{Z}}

\renewcommand{\C}{\mathbb{C}}

\newcommand{\unit}{1\!\!1}

\newcommand{\E}{\textsf{\upshape E}}
\newcommand{\prob}{\textsf{\upshape Pr}}

\newcommand{\norm}[1]{\left\lVert#1\right\rVert}

\newcommand{\wt}[1]{\widetilde{#1}}

\begin{document}

\title{Concentration and central limit theorem for the averaging process on $\Z^{d}$}

\author{Austin Eide\thanks{Department of Mathematics, Toronto Metropolitan University, Toronto, ON, Canada; e-mail: \texttt{austin.eide@torontomu.ca}}}

\date{}

\maketitle

\abstract{In the averaging process on a graph $G = (V, E)$, a random mass distribution $\upeta$ on $V$ is repeatedly updated via transformations of the form $\upeta_{v}, \upeta_{w} \mapsto (\upeta_{v} + \upeta_{w})/2$, with updates made according to independent Poisson clocks associated to the edge set $E$. We study the averaging process when $G$ is the integer lattice $\Z^{d}$. We prove that the process has tight asymptotic concentration around its mean in the $\ell^{1}$ and $\ell^{2}$ norms and use this to prove a central limit theorem. Previous work by Nagahata and Yoshida implies the central limit theorem when $d \geq 3$. Our results extend this to hold for all $d \geq 1$, and our techniques are likely applicable to other processes for which previously only the $d \geq 3$ case was tractable.}

\section{Introduction}

\subsection{Background}

In this paper, we consider the \textit{averaging process} $(\upeta_{t})$, $t \geq 0$, on the infinite integer lattice $\Z^{d}$, where we think of $\Z^{d}$ as a graph with edges $e = \{x,y\}$ given by pairs of vertices $x,y \in \Z^{d}$ with $|x-y| = \sum_{j=1}^{d}|x_{j} - y_{j}|=1$. The process is defined as follows. Each edge is associated with a Poisson process (``clock") of rate $1/2d$, independent of all other edges, and some initial probability distribution $\upeta$ on $\Z^{d}$ is fixed. Set $\upeta_{0} = \upeta$. If the clock on edge $e=\{x,y\}$ rings at time $t$, the vector $\upeta_{t^{-}}$ is updated by assigning $\upeta_{t}^{x},\upeta_{t}^{y}=(\upeta_{t^{-}}^{x} + \upeta_{t^{-}}^{y})/2$ and $\upeta_{t}^{w} = \upeta_{t^{-}}^{w}$ for all $w \neq x,y$. Brushing aside issues of existence for the moment, we remark that the dynamics of the averaging process on $\Z^{d}$ clearly conserve mass, so that $\upeta_{t}$ is a probability distribution for all $t$. The choice of $1/2d$ for the clock rates ensures that each vertex is involved in an average at rate $1$ (for any $d$).

The study of the averaging process (and its close relatives) has a relatively long history. On finite graphs, it dates back at least to a 1974 paper of DeGroot \cite{degroot1974reaching}, who used it as a model of consensus formation among interacting agents. More recently, considerable attention has been paid to the mixing of the averaging process on finite graphs---see \cite{aldous2012lecture}, \cite{caputo2023cutoff} \cite{chatterjee2022phase}, \cite{movassagh2022repeated}, \cite{quattropani2023mixing}, and \cite{spiro2022averaging}. 

The averaging process is an instance of what Liggett terms a \textit{linear system} in \cite{liggett1985interacting}, and the interested reader can find therein many results pertaining to  existence and ergodicity, including some alternative proofs of results in our Section \ref{construction_duality}. The present work is inspired by two papers of Nagahata and Yoshida (\cite{nagahata2009central}, \cite{nagahata2010note}), which establish a central limit theorem for a large class of linear systems (including the averaging process) on $\Z^{d}$ for $d \geq 3.$ Their results rely fundamentally on the transience of the simple random walk on $\Z^{d}$ in these dimensions. Using a refined approach, we are able to show a central limit theorem for the averaging process on $\Z^{d}$ simultaneously for all $d \geq 1$ (Theorem \ref{clt}). We also provide precise concentration estimates (Theorem \ref{two_norm_theorem}) which sharpen some results from (\cite{nagahata2009central}, \cite{nagahata2010note}). While we focus solely on the averaging process here, our methods are generalizable to other linear systems. 

In \cite{ferrari1998fluctuations} the authors consider a model similar to the averaging process which they call RAP---random average process. Our averaging process cannot strictly speaking be realized as an instance of RAP, but could with minor modifications to the definition. In \cite{ferrari1998fluctuations} and a more recent follow-up \cite{fontes2023scaling}, RAP is considered as a model of a dynamic random surface with heights indexed by points in $\Z^{d}$; when the initial condition is a hyperplane through the origin, the former paper establishes the existence of a distributional limit surface for the process, while the latter computes the distribution of each coordinate of the limiting surface. We were not able to easily bear out the relationship of the present work to the results of \cite{ferrari1998fluctuations} and \cite{fontes2023scaling}, but note there is some similarity. 

\subsection{Main results}

We are interested in asymptotic properties of the distribution $\upeta_{t}$. To motivate the main results, we recall a heuristic of Aldous and Lanoue from \cite{aldous2012lecture}. We can approximate the averaging process with a discrete analogue in which the mass of the distribution $\upeta_{t}$ is represented by a finite collection of tokens. When the clock on $\{x,y\}$ rings, the prescribed average can be achieved as follows: each vertex $x,y$ splits its pile of tokens (roughly) in half, chooses one of the halves randomly, and gives the selected half to the other vertex. 

Distinguish a special token by selecting it uniformly at random from all tokens before running the process, and let $W_{t}$ be its position at time $t$. If $\upeta$ is the initial condition, then $W_{0}$ is distributed according to $\upeta$. Given that $W_{t^{-}} = x$, if the clock on $\{x,y\}$ rings at time $t$, the special token moves to $y$ with probability $1/2$ and stays put with probability $1/2$. The dynamics of $(W_{t})$ are then the same as a continuous-time random walk on $\Z^{d}$ which, when at $x$, moves to any of its neighbors at rate $1/4d$, i.e., a continuous-time random walk run at ``half speed". If there are $N$ tokens in total, each performing their own random walks, we expect the proportion on vertex $x$ at time $t$ to be

\begin{equation}\label{heuristic}
	\E[\upeta_{t}^{x}] \approx \frac{\E[\#\{\text{tokens at }x\}]}{N} = \frac{N \cdot \prob_{\upeta}(W_{t} = x)}{N} = \prob_{\upeta}(W_{t} = x)
\end{equation}

where $\prob_{\upeta}(\cdot)$ denotes probability conditioned on $W_{0} \sim \upeta$. By the central limit theorem, for large $t$ the distribution of $W_{t}/\sqrt{t/2}$ is approximately $\mathcal{N}(\bm{0}, I_{d})$, a standard $d$-dimensional Gaussian. We show that the distribution $\upeta_{t}$ also has a Gaussian limit, in the following sense: 

\begin{theorem}\label{clt}
	Let $(\upeta_{t})$ be an averaging process on $\Z^{d}$ started from a probability distribution $\upeta$ on $\Z^{d}$. For any continuous, bounded function $f: \R^{d} \to \R$,
	
	$$
		\sum_{x \in \Z^{d}}f(x/\sqrt{t/2}\,)\upeta_{t}^{x} \xrightarrow{\prob} \int_{\R^{d}} f\,d\upnu
	$$
	
	as $t \to \infty$, where $\upnu$ is the $d$-dimensional standard Gaussian measure and $\xrightarrow{\prob}$ denotes convergence in probability.
\end{theorem}

In our presentation, Theorem \ref{clt} is a corollary of the following concentration result, in particular item (iii). In the statement, $| \cdot |$ and $\norm{\cdot}$ denote the $\ell^{1}$ and $\ell^{2}$ norms, respectively, on the vector space $\R^{\Z^{d}}$, and $\E^{\upeta}[\cdot]$ denotes expectation conditioned on the initial condition $\upeta_{0} = \upeta$.

\begin{theorem}\label{two_norm_theorem}
	Let $(\upeta_{t})$ be an averaging process on $\Z^{d}$ started from a probability distribution $\upeta$ on $\Z^{d}$. As $t \to \infty$,
	
		\begin{enumerate}[label = (\roman*)]
			
			\item $\E^{\upeta}\left[ \norm{\upeta_{t}}^{2} \right] \leq (1 + o(1))\frac{1}{(2\uppi t)^{d/2}}$
		
			\item $\E^{\upeta}\left[ \norm{\upeta_{t} - \E^{\upeta}[\upeta_{t}]}^{2} \right] = 
				\begin{cases} 
					O\left(\frac{1}{t^{(d+1)/2}} \right) \quad& d \neq 2\\
					O\left(\frac{1}{t\log^{2}t} \right) \quad& d = 2
				\end{cases}$
				
			\item $\E^{\upeta}\left[\,|\upeta_{t} - \E^{\upeta}[\upeta_{t}]|\,\right] = o(1)$
			
		\end{enumerate}
	
	with asymptotic equality holding in (i) when $\upeta$ is a unit mass on a single vertex.
		
\end{theorem}

An interpretation of the averaging process as a \textit{noisy heat flow} was suggested in \cite{sau2023concentration}. By the heuristic (\ref{heuristic}), in expectation the process evolves as $\E[\upeta_{t}] = A\E[\upeta_{t^{-}}]$ with $A = I - \frac{1}{2}L$ and $L$ the normalized Laplacian of $\Z^{d}$. This is equivalent to $\partial_{t}\E[\upeta_{t}] = -\frac{1}{2}L\E[\upeta_{t}]$, i.e., a heat flow on $\Z^{d}$. The process itself evolves as a stochastic flow $\upeta_{t} = A_{t}\upeta_{t^{-}}$, where $A_{t}$ is a random matrix determined by the Poisson clocks. The matrices $(A_{t})$ can be thought of as highly degenerate versions of $A$ which ``in expectation" equal $A$.

Thus (ii) and (iii) of Theorem \ref{two_norm_theorem} can be understood as measures of the extent to which the noise perturbs $(\upeta_{t})$ from the underlying heat flow. A concentration result of this type was proved in \cite{sau2023concentration} for the averaging process on the discrete torus $\mathbb{T}_{n}^{d}$. In that work, it was observed that the process has tight concentration around its mean long before the diffusive timescale $t = \Uptheta(n^{2})$ of the underlying heat flow, a phenomenon the author terms \textit{early concentration}. This fact is used to prove a hydrodynamic limit for the process on the torus $\mathbb{T}^{d}$. Our Theorem \ref{clt} can be seen as an analogue of this result, where the limit is now a density on $\R^{d}$. 

To see the significance of the concentration statements in Theorem \ref{two_norm_theorem}, suppose for simplicity that we begin the process from $\upeta$ a unit mass on the origin of $\Z^{d}$. The expectation $\E[\upeta_{t}]$ is the time-$t$ distribution of a continuous-time random walk on $\Z^{d}$ with jump rate $1/2$ started from the origin. Fix $r > 0$ and let $\mathcal{B}_{d}(r\sqrt{t})$ be the $\ell^{1}$ ball of radius $r\sqrt{t}$ centered at the origin. Note that $\text{Vol}(\mathcal{B}_{d}(r\sqrt{t})) = \Uptheta(t^{d/2})$ and the coordinates $\E[\upeta_{t}^{v}]$ are (uniformly) $\Uptheta(t^{-d/2})$ on $\mathcal{B}_{d}(r\sqrt{t})$. Consider the rescaled deviations $t^{d/2}\E[\,|\upeta_{t}^{v} - \E[\upeta_{t}^{v}]|\,]$. The average of this quantity over vertices in $\mathcal{B}_{d}(r\sqrt{t})$ is 

$$
	\frac{t^{d/2}}{\text{Vol}(\mathcal{B}_{d}(r\sqrt{t}))}\sum_{v \in \mathcal{B}_{d}(r\sqrt{t})}\E[\,|\upeta_{t}^{v} - \E[\upeta_{t}^{v}]|\,] = o(1)
$$

by (iii), suggesting that a ``typical" vertex in $\mathcal{B}_{d}(r\sqrt{t})$ has $\E[\,|\upeta_{t}^{v} - \E[\upeta_{t}^{v}]|\,] = o(\E[\upeta_{t}^{v}]).$
Thus the concentration is very tight for vertices which lie in the bulk of the distribution. In Section \ref{conclusion}, we give an example which demonstrates that concentration of this type is \textit{not} satisfied in general by processes whose mean trajectory is the heat flow. Giving sufficient conditions for a stochastic flow to have asymptotic concentration at the level of Theorem \ref{two_norm_theorem} is an interesting avenue for future research.

Our proof techniques are similar to those in \cite{sau2023concentration} (and numerous other works in this area) and rely primarily on the relationship between the averaging process and the random walk. The key idea is an extension of (\ref{heuristic}) which relates the joint moments $\E[\upeta_{t}^{x}\upeta_{t}^{y}]$ to pairs of coupled random walks on $\Z^{d}$. This relationship can be heuristically derived by considering the joint motion of two marked tokens. The tokens have dependencies when they occupy the same vertex or adjacent vertices, as in these cases they depend on at least one common edge clock. While these dependencies can be challenging to account for, in the case of the averaging process we are able to use analytic tools (specifically, singularity analysis of generating functions) to get precise enough results to yield the concentration we want.

\subsection{Outline \& notation}

In Section \ref{construction_duality} we give a general construction of the averaging process on graphs of bounded degree and establish its duality relationship with the continuous-time random walk. These results are more or less standard, though we include them here for the sake of completeness. In Section \ref{discrete_rw} we derive the transition rates of the ``two-token" coupled random walk associated to the averaging process on $\Z^{d}$ and establish some basic results about this walk and its discrete-time counterpart. Section \ref{gen_func} is devoted to the derivation of generating functions and singularity analysis. Finally, in Section \ref{proofs} we give proofs for Theorems \ref{clt} and \ref{two_norm_theorem}. Section \ref{conclusion} considers the potential for the techniques used in this paper to be applied to similar processes. There, we give a rough sketch that, when applied to the simple potlach process (aka meteor process) on $\Z^{d}$, our methods again give an exact asymptotic expression for $\E\left[\norm{\upeta_{t} - \E[\upeta_{t}]}^{2}\right]$, but we find that it is $\Uptheta(t^{-d/2})$ when initialized with worst case initial conditions, which is an asymptotically larger scale than in Theorem \ref{two_norm_theorem}.

For vectors $v$ with real entries indexed by a finite or countable index set $\mathcal{X}$, we generally let $v^{x}$ denote the $x$-coordinate of $v$. We write $|v| = \sum_{x \in \mathcal{X}}|v^{x}|$ for the $\ell^{1}$ norm of $v$ and $\norm{v} = \left(\sum_{x \in \mathcal{X}}(v^{x})^{2} \right)^{1/2}$ for the $\ell^{2}$ norm of $v$ (when these sums converge).

We let $\bm{0}$ denote the origin in $\Z^{d}$. For $r \geq 0$, we let 

$$
	\mathcal{B}_{d}(r) = \{x \in \Z^{d}\,:\,|x| \leq r\}
$$

be the $\ell^{1}$ ball of radius $r$ in $\Z^{d}$ centered at $\bm{0}$. When $r$ is an integer, we also define the $\ell^{1}$ sphere

$$
	\mathcal{S}_{d}(r) = \{x \in \Z^{d}\,:\,|x| = r\}.
$$

In the body of the paper we consider random processes in both continuous and discrete time. We reserve the variable $t$ for a continuous time index and $n$ for a discrete time index. Unless stated otherwise, we assume $t$ ranges over the interval $\R_{\geq 0}$ and $n$ over $\Z_{\geq 0}$.

Given a graph $G$ and a probability distribution $\upeta$ on the vertices of $G$, we use the notation $\prob_{\upeta}(\cdot)$ and $\E_{\upeta}[\cdot]$ to denote probabilities and expectations corresponding to a random walk on $G$ conditioned to have initial state distributed as $\upeta$. When $\upeta$ is concentrated on a single vertex $v$, we simply write $\prob_{v}(\cdot)$ and $\E_{v}[\cdot]$. 

If $(\upeta_{t})$ is an averaging process on $G$ with initial state $\upeta_{0} = \upeta$, we write $\prob^{\upeta}(\cdot)$ and $\E^{\upeta}[\cdot]$ to denote probabilities and expectations corresponding to $(\upeta_{t})$. When there is no risk of misunderstanding, we suppress superscripts and subscripts.

Given a probability distribution $\mathcal{D}$, we write $X \sim \mathcal{D}$ to mean that the random variable $X$ is distributed according to $\mathcal{D}$. We use $\text{Ber}(\cdot)$ and $\text{Po}(\cdot)$ to denote the Bernoulli and Poisson distributions, respectively, with their pertinent parameters.

\section{Construction and duality}\label{construction_duality}

Here we give a construction of the averaging process on general graphs satisfying a bounded degree conditoin using a version of the Harris graphical construction (\cite{harris1972nearest},\cite{harris1978additive}), and also establish its duality relationship with the random walk. Our presentation is a direct adaptation of those given in \cite{burdzy2015meteor} and \cite{swart2022course}. An alternative construction via semigroup theory is given in \cite[Chapter IX]{liggett1985interacting}. Still another construction, via the theory of stochastic differential equations, can be undertaken along the lines of \cite{nagahata2010note}. We opt for the graphical construction for the sake of clarity.

Let $V$ be a finite or countably infinite vertex set. We assign a symmetric weight $r_{vw} = r_{wv} \geq 0$ to each pair of distinct vertices $\{v,w\}$ and define the edge set $E$ to be the collection of pairs $\{v,w\}$ such that $r_{vw} > 0$. Then we have a weighted graph $G = (V,E,r)$. Strictly speaking, identifying the edge set $E$ is unnecessary, but we find it useful in some proofs.

We assume the weights are such that 

\begin{equation}\label{bdd_degree}
	\Updelta = \sup_{v \in V}\sum_{w \in V \setminus \{v\}}r_{vw} < \infty
\end{equation}

For each $e \in E$, let $(N_{t}^{e})$ be a rate-$r_{e}$ Poisson process on $\R_{\geq 0}$, independent of all $(N_{t}^{e'})$ with $e' \neq e$. The variable $N_{t}^{e}$ counts the number of clock rings on $e$ up to time $t$. Given an initial probability distribution $\upeta$ on $V$, the averaging process $(\upeta_{t})$ on $G$ is defined in terms of the Poisson processes $\{(N_{t}^{e})\}_{e \in E}$ as follows. Let $\upeta_{0} = \upeta$; for each edge $e = \{v,w\}$, at each $t > 0$ such that $N_{t}^{e} - N_{t^{-}}^{e} > 0$, we set $\upeta_{t}^{v}, \upeta_{t}^{w} = (\upeta_{t^{-}}^{v} + \upeta_{t^{-}}^{w})/2$ (and leave all other coordinates unchanged).

When $V$ is finite, the informal description of $(\upeta_{t})$ above suffices. When $V$ is countably infinite it is not clear \textit{a priori} that things are well-defined. In general, we must keep track of an infinite collection of Poisson processes, which means we will observe accumulation points a.s. But at such points, it is not possible (from a global perspective) to determine a precise order in which to perform the averages. The graphical construction gets around this issue by establishing that, over finite time intervals, the vertices of $G$ can be separated into distinct finite components on which separate averaging processes can be performed independently in a well-defined manner. 

\subsection{Construction}

Existence and uniqueness of the averaging process is established in the next proposition.

\begin{proposition}\label{existence_uniqueness} Let $G = (V, E, r)$ be a weighted graph satisfying (\ref{bdd_degree}). Given an initial distribution $\upeta$ on $V$, there is a unique process $(\upeta_{t})$ with $\upeta_{0} = \upeta$ satisfying the description given above.
\end{proposition}

\begin{proof} 

For a vertex $v \in V$ and $t \geq 0$, we define a $(v,t)$-influence path as a map $\upphi: [0, t] \to V$ satisfying  

	\begin{enumerate}[label=(\roman*)]
		\item $\upphi_{t} = v$, and
		\item for all $s \in [0,t]$, if $\upphi_{s} \neq \upphi_{s^{-}}$, then the clock on edge $\{\upphi_{s^{-}}, w\}$ rings at time $s$ for some $w$.
	\end{enumerate}
	
Let $I_{t}^{v}$ be the set of vertices $w$ such that there is $(v,t)$-influence path $\upphi$ with $\upphi_{0} = w$. Intuitively, the evolution of the coordinate function $\upeta_{s}^{v}$ for $0 \leq s \leq t$ depends only on the coordinates corresponding to vertices in $I_{t}^{v}$.

\textbf{Claim:} Almost surely, $|I_{t}^{v}| < \infty$ for all $v \in V$ and $t \geq 0$. 

Given the claim, we show how to define the coordinate $\upeta_{t}^{v}$ for a fixed $v \in V$ and $t \geq 0$. Let $E_{t}^{v}$ be the set of edges $e$ such that $e$ contains at least one vertex in $I_{t}^{v}$ and the clock on $e$ rings at least once over the interval $[0,t]$. By (\ref{bdd_degree}), we have that, almost surely, for every $t \geq 0$, each $w \in V$ experiences at most finitely many clock rings on its incident edges during $[0,t]$. This fact, combined with the fact that $I_{t}^{v}$ is finite a.s.\ by the claim, implies that $E_{t}^{v}$ is finite a.s., and hence we may enumerate $E_{t}^{v}$, from earliest ring to latest and possibly with repetition, as a finite sequence $e_{1}, e_{2}, \dots, e_{T}$.

Now, for each $e \in E$, define the matrix $A^{e} \in \R^{V \times V}$ by 

\begin{equation}\label{matrix_def}
A^{e}(u,w) = \begin{cases}
\frac{1}{2} \quad& e = \{u,w\}\\
1 \quad& u=w \text{ and }u \not\in e\\
0 \quad& \text{otherwise}
\end{cases}
\end{equation}

Multiplication of a vector in $\R^{V}$ by $A^{e}$ replaces the $u$ and $w$ coordinates with their average and leaves all other coordinates unchanged. Then with $A = A^{e_{T}}A^{e_{T-1}} \cdots A^{e_{1}}$, we define $\upeta_{t}^{v} = (A\upeta)^{v}$. 

This construction works for any $v$ and $t$, and it is straightforward to verify that it gives a well-defined map from $[0, \infty)$ to $\R^{\Z^{d}}$ which adheres to the informal description of the process given above, and moreover is the unique map which does so.

We finish by proving the claim. First, note that it follows from the definition of $(v,t)$-influence path that $I_{s}^{v} \subseteq I_{t}^{v}$ for $s \leq t$. In particular, $I_{t}^{v} \subseteq I_{\lceil t \rceil}^{v}$. Thus it suffices to show that $\prob(|I_{t}^{v}| = \infty) = 0$ for an arbitrary $v \in V$ and $t \in \Z_{\geq 0}$; successive union bounds over all $v \in V$ and $t \in \Z_{\geq 0}$ give the full result.

Define $B_{0} = \{v\}$ and $\uptau_{0} = t$. For $k \geq 1$, given $B_{k-1}$ and $\uptau_{k-1}$, let $\uptau_{k}$ be the time of the latest clock ring before time $\uptau_{k-1}$ which comes on some edge $\{u,w\}$ with $u \in B_{k-1}$ and $w \not\in B_{k-1}$, or let $\uptau_{k} = 0$ if no such ring occurs. Let $B_{k} = B_{k-1} \cup \{w\}$ if $\uptau_{k} > 0$, or $B_{k} = B_{k-1}$ if $\uptau_{k} = 0$. Let $B = \bigcup_{k=0}^{\infty}B_{k}$, and observe that we have the inclusion $I_{t}^{v} \subseteq B$. We will show that $B$ is finite, almost surely.

By our construction, for $k = 0,1,2,\dots$, if $\uptau_{k}  > 0$, then the set $B_{k}$ contains exactly $k+1$ vertices. Using this and (\ref{bdd_degree}), we find that, for as long as $\uptau_{k} > 0$, the increments $\{\uptau_{k-1} - \uptau_{k}\}$, $k = 1,2,3,\dots,$ stochastically dominate a sequence of independent exponential random variables $\{Z_{k}\}$, $k = 1,2,3,\dots$, where $Z_{k}$ has rate $\Updelta k$. Then we have

\begin{eqnarray}
	\prob(|B| \geq N) &\leq& \prob\left(\sum_{k=1}^{N-1}(\uptau_{k-1} - \uptau_{k}) < t \right)\nonumber\\
	&\leq& \prob\left(\sum_{k=1}^{N-1}Z_{k} < t \right) \label{to_bound}
\end{eqnarray}

The sum $\sum_{k=1}^{N-1}Z_{k}$ has mean $\frac{1}{\Updelta}\sum_{k=1}^{N-1}\frac{1}{k} = \upomega(1)$ and variance $\frac{1}{\Updelta^{2}}\sum_{k=1}^{N-1}\frac{1}{k^{2}} = O(1)$ as $N \to \infty$. An application of Chebyshev's inequality then gives that the probability (\ref{to_bound}) is $o(1)$ as $N \to \infty$. In particular, this implies that $B$ is finite almost surely, completing the proof.

\end{proof}

When $(\upeta_{t})$ is started from an initial condition which is a unit mass on vertex $v$, we will use the notation $(\upeta_{t}) = (\upeta_{t}(v,\cdot))$, so that $\upeta_{t}(v,w)$ is the amount of mass at vertex $w$ at time $t$. This notation, which we borrow from \cite{sau2023concentration}, is well-suited to the interpretation of $(\upeta_{t})$ as a noisy heat flow. Indeed, from the construction in the proof of Proposition \ref{existence_uniqueness}, it straightforward to see that we can couple the main averaging process $(\upeta_{t})$ with a collection of averaging processes $\{\upeta_{t}(v,\cdot)\}_{v \in V}$ in such a way that, if the initial condition is $\upeta_{0} = \upeta$, 

\begin{equation}\label{noisy_heat_kernel}
	\upeta_{t}^{w} = \sum_{v \in V}\upeta^{v}\upeta_{t}(v,w) \quad\forall\,w \in V,\, t \geq 0.
\end{equation}

In this sense, we can think of $\upeta_{t}(\cdot, \cdot)$ as a noisy heat kernel. We will use the decomposition (\ref{noisy_heat_kernel}) extensively to reduce the study of general initial conditions to the study of point mass initial conditions. Since the initial condition is implicit in $\upeta_{t}(v,\cdot)$, we do not use superscripts when writing expectations with respect to these processes.

\subsection{Duality with the random walk}

Now we establish the relationship between the averaging process on $G$ and the continuous-time random walk on $G$. Let $B_{1}, B_{2}, B_{3}, \dots$ be a sequence of i.i.d.\ $\text{Ber}(1/2)$ random variables, independent of the edge clocks. We define a continuous-time random walk $(W_{t})$ on $G$ in terms of the edge clocks and the $B_{j}$'s as follows. Suppose that $W_{t} = v$ for some $t \geq 0$, and that the walk has made exactly $j$ jumps thus far. The walk waits for the first clock on some incident edge $e = \{v,w\}$ to ring and then reveals the value of $B_{j+1}$, moving to $w$ if $B_{j+1} = 1$ and staying at $v$ if $B_{j+1} = 0$. 

The process $(W_{t})$ is a continuous-time random walk on $G$ with transition rates given by $v \xrightarrow{r_{vw}/2} w$ for all $v,w \in V$, where $\xrightarrow{r}$ denotes a transition of rate $r$. Since both $(\upeta_{t})$ and $(W_{t})$ are defined in terms of the edge clocks, we may evolve the processes simultaneously as a joint process on the same probability space. We use this joint evolution to prove (\ref{fm}) of Proposition \ref{duality_lemma} below.

To prove (\ref{sm}) of Proposition \ref{duality_lemma}, we use the following extension. Let $B_{1}^{(1)}, B_{2}^{(1)}, B_{3}^{(1)}, \dots$ and $B_{1}^{(2)}, B_{2}^{(2)}, B_{3}^{(2)}, \dots$ be a pair of i.i.d.\ $\text{Ber}(1/2)$ sequences, independent of one another and the edge clocks. Let $(\wt{W}_{t}^{(1)}, \wt{W}_{t}^{(2)})$ be a pair of continuous-time random walks, each with the same distribution as $(W_{t})$, defined in terms of the same edge clocks but such that $(\wt{W}_{t}^{(1)})$ depends on $\{B_{j}^{(1)}\}_{j=1}^{\infty}$ and $(\wt{W}_{t}^{(2)})$ depends on $\{B_{j}^{(2)}\}_{j=1}^{\infty}$. The walks have non-trivial dependence, since whenever $\wt{W}_{t}^{(1)}$ and $\wt{W}_{t}^{(2)}$ are within distance 1 in the graph they depend on at least one common edge clock. At distances greater than 1 apart the walks move independently.

The dynamics of $(\wt{W}_{t}^{(1)}, \wt{W}_{t}^{(2)})$ are given below. Suppose $u,v,w$ are distinct vertices. (Compare with \cite[Section 2.4]{aldous2012lecture}.)

\begin{equation}\label{coupled_rates}
	\begin{aligned}
		(u,v) \xrightarrow{r_{vw}/2}& (u,w)\\
		(u,v) \xrightarrow{r_{uw}/2}& (w,v)\\
		(u,v) \xrightarrow{r_{uv}/4}& (v,v)\\
		(u,v) \xrightarrow{r_{uv}/4}& (u,u)\\
		(u,v) \xrightarrow{r_{uv}/4}& (v,u)\\
		(u,u) \xrightarrow{r_{uv}/4}& (u,v)\\
		(u,u) \xrightarrow{r_{uv}/4}& (v,u)\\
		(u,u) \xrightarrow{r_{uv}/4}& (v,v)
	\end{aligned}
\end{equation}
	
A transition from, say, $(u,v)$ to $(u,u)$ requires i) the clock on $\{u,v\}$ to ring (occurs at rate $r_{uv}$), and ii) upon the clock ring, $\wt{W}_{t}^{(1)}$ must stay at $u$ and $\wt{W}_{t}^{(2)}$ must jump to $u$ (occurs with probability $\frac{1}{2} \cdot \frac{1}{2} = \frac{1}{4}$). The conjunction of these events can be modeled with a single exponential clock of rate $r_{uv}/4$, which explains the transition rate $(u,v) \xrightarrow{r_{uv}/4} (u,u).$ Note the differences between $(\wt{W}_{t}^{(1)}, \wt{W}_{t}^{(2)})$ and a pair of independent walks $(W_{t}^{(1)}, W_{t}^{(2)})$ with the same marginals---the independent walks, for instance, never perform a ``swap" step $(u,v) \to (v,u)$ almost surely, while the coupled walks make this step at a positive rate.

The main duality relationship is below. Given a probability distribution $\upeta$ on $V$, we write $\upeta \otimes \upeta$ for the product distribution on $V \times V$ where each factor is drawn independently from $\upeta$. 

\begin{proposition}\label{duality_lemma}

Let $\upeta$ be a probability distribution on $G$ and $(\upeta_{t})$ be an averaging process on $G$. Let $(W_{t})$ and $(\wt{W}_{t}^{(1)}, \wt{W}_{t}^{(2)})$ be as defined above. Then for all $v \in V$ and $t \geq 0$, 

\begin{equation}\label{fm}
	\E^{\upeta}[\upeta_{t}^{v}] = \prob_{\upeta}(W_{t} = v)
\end{equation}

and for all $v,w \in V$ (not necessarily distinct) and $t \geq 0$,

\begin{equation}\label{sm}
	\E^{\upeta} [\upeta_{t}^{v}\upeta_{t}^{w}]= \prob_{\upeta \otimes \upeta}(\wt{W}_{t}^{(1)} = v,\,\wt{W}_{t}^{(2)} = w).
\end{equation}

\end{proposition}

\begin{proof}

For clarity of presentation we will assume throughout the proof that the initial distributions $W_{0}$ and $(\wt{W}_{0}^{(1)}, \wt{W}_{0}^{(2)})$ are as given in (\ref{fm}) and (\ref{sm}), and don't include this information explicitly in the notation.

Let $(\mathcal{F}_{t})$, $t \geq 0$, be the filtration associated to the collection of edge clocks. We will show that for all $v,w \in V$ and $t \geq 0$, we have

\begin{equation}\label{fm_conditional}
\prob(W_{t} = v\,|\,\mathcal{F}_{t})=\upeta_{t}^{v} \quad\text{a.s.}
\end{equation}

and

\begin{equation}\label{sm_conditional}
\prob(\wt{W}_{t}^{(1)} = v,\,\wt{W}_{t}^{(2)} = w\,|\,\mathcal{F}_{t}) = \upeta_{t}^{v}\upeta_{t}^{w} \quad\text{a.s.}
\end{equation}

Applying expectations to both sides of (\ref{fm_conditional}) and (\ref{sm_conditional}) then gives (\ref{fm}) and (\ref{sm}), respectively. In fact, it will suffice to show (\ref{fm_conditional}), since it implies (\ref{sm_conditional}). This is due to the fact that $\wt{W}_{t}^{(1)}$ and $\wt{W}_{t}^{(2)}$ are conditionally independent given $\mathcal{F}_{t}$, that is,

$$
	\prob(\wt{W}_{t}^{(1)} = v,\,\wt{W}_{t}^{(2)} = w\,|\,\mathcal{F}_{t}) = \prob(\wt{W}_{t}^{(1)} = v\,|\,\mathcal{F}_{t})\cdot\prob(\wt{W}_{t}^{(2)} = w\,|\,\mathcal{F}_{t}).
$$

Since $(\wt{W}_{t}^{(i)})$ has the same distribution as $(W_{t})$ for each $i=1,2$, the right-hand side above is $\upeta_{t}^{v}\upeta_{t}^{w}$ a.s.\ by (\ref{fm_conditional})

Before proving (\ref{fm_conditional}), we briefly recall how $\upeta_{t}^{v}$ is defined for a given $v \in V$ and $t \geq 0$. In the proof of Proposition \ref{existence_uniqueness}, we show that $I_{t}^{v}$, the set of vertices in $G$ which lie on some $(v,t)$-influence path, is finite, almost surely. The set of relevant edges $E_{t}^{v}$ is then finite, and we may enumerate these edges, from earliest to latest ring, as a sequence $e_{1},e_{2},\dots,e_{T}$, possibly with repetition. We then define $\upeta_{t}^{v} = (A\upeta)^{v}$, where $A = A^{e_{T}}A^{e_{T-1}} \cdots A^{e_{1}}$ and the matrices $A^{e}$ are as defined in (\ref{matrix_def}).

Now, let $v \in V$ and $t \geq 0$ be fixed. We will define a modified walk $(\widehat{W}_{s})$ in the conditional probability space obtained by conditioning on $\mathcal{F}_{t}$. The modified walk has the same initial distribution and dynamics as $(W_{s})$, but $(\widehat{W}_{s})$ can only jump upon rings on edges in $E_{t}^{v}$; all other edge rings are disregarded. (The indicator variables $\unit\{e \in E_{t}^{v}\}$ are $\mathcal{F}_{t}$-measurable, hence the set $E_{t}^{v}$ is known after conditioning on $\mathcal{F}_{t}$ and this is a legitimate definition of $(\widehat{W}_{s})$.) 

The definition of $(\widehat{W}_{s})$ depends on a fixed vertex $v$, but we omit this from the notation for clarity. For consistency, we write $\prob(\widehat{W}_{s} = w\,|\,\mathcal{F}_{t})$ for the time-$s$ distribution of the modified walk, even though the walk is only defined in the $\mathcal{F}_{t}$-conditional probability space, so the notation is technically redundant.

Item (\ref{fm_conditional}) clearly follows from the following two statements:

	\begin{enumerate}[label = (\roman*)]
		\item $\prob(\widehat{W}_{t} = v\,|\,\mathcal{F}_{t}) = \upeta_{t}^{v}$ a.s., and 
		\item $\prob(W_{t} = v\,|\,\mathcal{F}_{t}) = \prob(\widehat{W}_{t} = v\,|\,\mathcal{F}_{t})$ a.s.
	\end{enumerate}

We begin with (i). Let $e_{1},e_{2},\dots,e_{T}$ be the sequence of relevant edges. Let $t_{0} =0$ and define the times $t_{0} < t_{1} < t_{2} < \cdots < t_{T} <  t$ such that the ring corresponding to $e_{j}$ occurs at time $t_{j}$ for $j = 1,\dots,T$. For any $j$, for each $w \in V$ we have that $\prob(\widehat{W}_{s} =w\,|\,\mathcal{F}_{t})$ is constant over $s \in [t_{j-1}, t_{j})$, since either no rings occur during this span (in the case of $j=1$), or exactly one ring occurs during this span, at $t_{j-1}$ (in the case $j \geq 2$). Further, $\prob(\widehat{W}_{s} = w\,|\,\mathcal{F}_{t}) = \prob(\widehat{W}_{t_{T}} = w\,|\,\mathcal{F}_{t})$ for all $s \in [t_{T}, t]$, since the only ring during this span is at $t_{T}$. 

Thus, the distribution of $(\widehat{W}_{s})$ has a discrete evolution over the interval $[0,t]$ with exactly $T$ jumps. We realize this evolution as a sequence of distributions $\upzeta_{0},\upzeta_{1},\dots,\upzeta_{T} \in \R^{V}$, with $\upzeta_{j}$ defined by $\upzeta_{j}^{w} = \prob(\widehat{W}_{t_{j}} = w\,|\,\mathcal{F}_{t})$ for each $w \in V$. We have $\upzeta_{0} = \upeta$ since $\widehat{W}_{0}$ is distributed according to $\upeta$. We claim that $\upzeta_{j} = A^{e_{j}}\upzeta_{j-1}$ for each $j = 1, \dots, T$. Indeed, we see

$$
\upzeta_{j}^{w} = \begin{cases} \frac{1}{2}\upzeta_{j-1}^{u} + \frac{1}{2}\upzeta_{j-1}^{w} \quad&\text{if }e_{j} = \{u,w\} \text{ for some }u \\ \upzeta_{j-1}^{w}\quad&\text{otherwise} \end{cases}
$$

which is equivalent to $\upzeta_{j} = A^{e_{j}}\upzeta_{j-1}$. It follows that 

$$
	\upzeta_{T} = A^{e_{T}}A^{e_{T-1}} \cdots A^{e_{1}}\upzeta_{0} = A^{e_{T}}A^{e_{T-1}} \cdots A^{e_{1}}\upeta 
$$ 

and in particular $\upzeta_{T} = \prob(\widehat{W}_{t} = v\,|\,\mathcal{F}_{t}) = (A\upeta)^{v} = \upeta_{t}^{v}$, proving (i).

For (ii), we exhibit a coupling of the walks $(W_{s})$ and $(\widehat{W}_{s})$ in the $\mathcal{F}_{t}$-conditional probability space so that $W_{t} = v$ if and only if $\widehat{W}_{t} = v$. Note that if $W_{s} \in V \setminus I_{t}^{v}$ for any $0 \leq s \leq t$, then $\prob(W_{t} = v\,|\,\mathcal{F}_{t}) = 0$, and likewise for the modified walk $(\widehat{W}_{s})$, due to the fact that any trajectory of either walk which ends at $v$ is necessarily a $(v,t)$-influence path. Furthermore, if $W_{0} = \widehat{W}_{0} = w \in I_{t}^{v}$, then it is easy to see that the walks can be coupled so that $W_{s} = \widehat{W}_{s}$ for as long as they stay within $I_{t}^{v}$---simply use the same ``stay-or-go" decisions for both walks.

Since $W_{0}$ and $\widehat{W}_{0}$ are both distributed according to $\upeta$, we may couple the initial states so that $W_{0} = \widehat{W}_{0} = w$. If $w \in I_{t}^{v}$, then we run the coupling described above, until either time $t$ or until the walks leave $I_{t}^{v}$. If $w \not\in I_{t}^{v}$, we impose no coupling on the walks, but remark that in this case neither walk can be at $v$ at time $t$. In this way, we have that $W_{t} = v$ if and only if $\widehat{W}_{t} = v$, and the proof is complete. 
\end{proof}

The proof of Proposition \ref{duality_lemma} in fact gives a bit more: by summing both sides of (\ref{fm_conditional}) over $V$, we get that for any $t \geq 0$

\begin{equation}\label{conservative}
	\sum_{v \in V}\upeta_{t}^{v} = 1\,\quad\text{a.s.}
\end{equation}

That is, the process is \textit{conservative}. While the informal description of the dynamics of $(\upeta_{t})$ suggest as much, it is perhaps not immediately clear from the construction given in the proof of Proposition \ref{existence_uniqueness}. In particular, when $(\upeta_{t})$ is started from a probability distribution, (\ref{conservative}) ensures that the infinite series defining quantities such as $\norm{\upeta_{t}}^{2}$ and $\norm{\upeta_{t} - \E^{\upeta}[\upeta_{t}}^{2}$ converge a.s., which we use implicitly in Corollary \ref{two_norm_identity1} below. This corollary gives the fundamental relationship between the $\ell^{2}$ concentration of the averaging process on $G$ and the random walk on $G$.

\begin{corollary}\label{two_norm_identity1} Let $\upeta$, $(\upeta_{t})$, and $(W_{t})$ be as above. Let $(W_{t}^{(1)}, W_{t}^{(2)})$ be a pair of independent copies of $(W_{t})$, and let $(\wt{W}_{t}^{(1)}, \wt{W}_{t}^{(2)})$ be a pair of copies of $(W_{t})$ coupled according to (\ref{coupled_rates}). Then for all $t \geq 0$,

	\begin{enumerate}[label = (\roman*)]
	
		\item $\E^{\upeta}\left[\norm{\upeta_{t}}^{2} \right] = \prob_{\upeta \otimes \upeta}(\wt{W}_{t}^{(1)} = \wt{W}_{t}^{(2)})$,
		
		\item $\norm{\E^{\upeta}[\upeta_{t}]}^{2} = \prob_{\upeta \otimes \upeta}(W_{t}^{(1)} = W_{t}^{(2)})$, and
	
		\item $\E^{\upeta}\left[\norm{\upeta_{t} -\E^{\upeta}[\upeta_{t}]}^{2}\right] = \prob_{\upeta \otimes \upeta}(\wt{W}_{t}^{(1)} = \wt{W}_{t}^{(2)}) - \prob_{\upeta \otimes \upeta}(W_{t}^{(1)} = W_{t}^{(2)}).$
	
	\end{enumerate}
	
\end{corollary} 

\begin{proof}

	We have
	
	$$
		\norm{\E^{\upeta}[\upeta_{t}]}^{2} = \sum_{x \in Z^{d}}(\E^{\upeta}[\upeta_{t}^{v}])^{2}
	$$
	
	and
	
	$$
		\E^{\upeta}\left[\norm{\upeta_{t}}^{2} \right] = \E^{\upeta}\left[ \sum_{x \in \Z^{d}}(\upeta_{t}^{v})^{2}\right] = \sum_{x \in \Z^{d}}\E^{\upeta}[(\upeta_{t}^{v})^{2}]
	$$
	
	where in the latter we use the dominated convergence theorem to interchange expectation and sum. These are, respectively, $\prob_{\upeta \otimes \upeta}(W_{t}^{(1)} = W_{t}^{(2)})$ and $ \prob_{\upeta \otimes \upeta}(\wt{W}_{t}^{(1)} = \wt{W}_{t}^{(2)})$ by Proposition \ref{duality_lemma}, establishing (i) and (ii). A simple computation gives
	
	$$
		\E^{\upeta}\left[\norm{\upeta_{t} -\E^{\upeta}[\upeta_{t}]}^{2}\right] = \E^{\upeta}\left[\norm{\upeta_{t}}^{2} \right] - \norm{\E^{\upeta}[\upeta_{t}]}^{2}
	$$
	
	so that (iii) follows from (i) and (ii).
	
\end{proof}

\section{Random walks on $\Z^{d}$}\label{discrete_rw}

Henceforth, we equip the edges of $\Z^{d}$ with independent Poisson clocks of rate $1/2d$, and let $(\upeta_{t})$ be the corresponding averaging process. Let $(W_{t})$ be a continuous-time random walk on $\Z^{d}$ with transition rates given by $x \xrightarrow{1/4d} y$ for all pairs $x,y$ with $|x-y| = 1$ (and rate $0$ otherwise). Let $(W_{t}^{(1)}, W_{t}^{(2)})$ be a pair of independent copies of $(W_{t})$, and let $(\wt{W}_{t}^{(1)}, \wt{W}_{t}^{(2)})$ be a pair of copies of $(W_{t})$ coupled according to the rates given in (\ref{coupled_rates}). 

We define the \textit{heat kernel} $(h_{t})$ for $t \geq 0$ by

$$
	h_{t}(x,y) = \prob_{x}(W_{t} = y)
$$

for all $x,y \in \Z^{d}$ and note in particular that $\E[\upeta_{t}(x,y)] = h_{t}(x,y)$ by Proposition \ref{duality_lemma}. (Recall that $(\upeta_{t}(x, \cdot))$ denotes an averaging process started from a unit mass on $x$.)

\subsection{The difference walks}

In light of Corollary \ref{two_norm_identity1}, it will be beneficial for us to study the events $W_{t}^{(1)} = W_{t}^{(2)}$ and $\wt{W}_{t}^{(1)}=\wt{W}_{t}^{(2)}$. For $t \geq 0$ we let 

$$
	D_{t} = W_{t}^{(1)} - W_{t}^{(2)}\quad\text{ and }\quad\wt{D}_{t} = \wt{W}_{t}^{(1)} - \wt{W}_{t}^{(2)},
$$

where arithmetic is performed coordinate-wise. The difference processes $(D_{t})$ and $(\wt{D}_{t})$ are both continuous-time random walks on $\Z^{d}$, and note that $D_{t} = \bm{0}$ (resp.\ $\wt{D}_{t} = \bm{0}$) iff $W_{t}^{(1)} = W_{t}^{(2)}$ (resp. $\wt{W}_{t}^{(1)} = \wt{W}_{t}^{(2)}$). It is not hard to see that $(D_{t})$ is in fact a simple random walk with jump rate $1$. $(\wt{D}_{t})$ behaves like $(D_{t})$ outside of $\mathcal{B}_{d}(1)$, but has perturbed dynamics inside of $\mathcal{B}_{d}(1)$ due to the dependency between $\wt{W}_{t}^{(1)}$ and $\wt{W}_{t}^{(2)}$ within distance $1$. We derive the exact dynamics of $(\wt{D}_{t})$ in Section \ref{dynamics}.

The next lemma demonstrates the utility of the difference processes. In particular, it allows us to reduce (i) and (ii) of Theorem \ref{two_norm_theorem} to the problem of estimating the transition probabilities of $(\wt{D}_{t})$ and $(D_{t})$ when each walk is started from the origin.

\begin{lemma}\label{difference_lemma}
	For any probability distribution $\upeta$ on $\Z^{d}$, for all $t \geq 0$
	
	\begin{enumerate}[label = (\roman*)]
	
		\item $\E^{\upeta}\left[ \norm{\upeta_{t}}^{2} \right] \leq \prob_{\bm{0}}(\wt{D}_{t} = \bm{0})$, and 
		
		\item $\E^{\upeta}\left[\norm{\upeta_{t} - \E^{\upeta}[\upeta_{t}]}^{2} \right] \leq \prob_{\bm{0}}(\wt{D}_{t} = \bm{0}) - \prob_{\bm{0}}(D_{t} = \bm{0}).$
		
	\end{enumerate}
	
	with equality in both (i) and (ii) whenever $\upeta$ is a unit mass on a single vertex.
\end{lemma}

\begin{proof}

Using the decomposition (\ref{noisy_heat_kernel}) we may express
	
	\begin{eqnarray}
		\norm{\upeta_{t} - \E^{\upeta}[\upeta_{t}]}^{2} &=& \sum_{y \in \Z^{d}}(\upeta_{t}^{y} - \E^{\upeta}[\upeta_{t}^{y}])^{2}\nonumber\\
		&=&\sum_{y \in \Z^{d}}\left(\sum_{x \in \Z^{d}}\upeta^{x}(\upeta_{t}(x,y) - h_{t}(x,y)) \right)^{2}\nonumber\\
		&\leq&\sum_{y \in \Z^{d}}\sum_{x \in \Z^{d}}\upeta^{x}(\upeta_{t}(x,y) - h_{t}(x,y))^{2}\nonumber\\
		&=&\sum_{x \in \Z^{d}}\upeta^{x}\norm{\upeta_{t}(x,\cdot) - h_{t}(x,\cdot)}^{2}\label{kernel_dev1}
	\end{eqnarray}
	
	where the third line is an application of Jensen's inequality; note that we get an equality in that line whenever $\upeta$ is a unit mass on some $x \in \Z^{d}$. By Corollary \ref{two_norm_identity1} (iii), for any $x$ we have
	
	$$
		\E\left[ \norm{\upeta_{t}(x,\cdot) - h_{t}(x,\cdot)}^{2}\right] = \prob_{x,x}(\wt{W}_{t}^{(1)} = \wt{W}_{t}^{(2)}) - \prob_{x,x}(W_{t}^{(1)}=W_{t}^{(2)})
	$$
	
	where the subscript $x,x$ denotes that both walks in the pair are started from $x$. Observe
	
	$$
		 \prob_{x,x}(\wt{W}_{t}^{(1)} = \wt{W}_{t}^{(2)}) = \prob_{\bm{0}}(\wt{D}_{t} = \bm{0}) \quad\text{ and } \quad \prob_{x,x}(W_{t}^{(1)}, W_{t}^{(2)}) = \prob_{\bm{0}}(D_{t} = \bm{0}).
	$$
	
	Thus, applying expectation to (\ref{kernel_dev1}) and using the dominated convergence theorem to interchange expectation and sum gives
	
	$$
		\E^{\upeta}\left[ \norm{\upeta_{t} - \E^{\upeta}[\upeta_{t}]}^{2} \right] \leq  \prob_{\bm{0}}(\wt{D}_{t} = \bm{0})-\prob_{\bm{0}}(D_{t} = \bm{0})
	$$
	
	again with equality whenever $\upeta$ is concentrated on a single vertex. A similar computation gives 
	
	$$
		\E^{\upeta}\left[ \norm{\upeta_{t}}^{2} \right] \leq \prob_{\bm{0}}(\wt{D}_{t} = \bm{0})
	$$
	
	with the same equality condition as above.
	
\end{proof}

Our goal for the rest of the paper is essentially to estimate $\prob_{\bm{0}}(\wt{D}_{t} = \bm{0})$. A standard local limit theorem gives $\prob_{\bm{0}}(D_{t} = \bm{0}) = (1+o(1))\frac{1}{(2\uppi t)^{d/2}}$. From Lemma \ref{difference_lemma}, we can see that, in order to obtain the $\ell^{2}$ concentration required by Theorem \ref{two_norm_theorem} (i.e., at the scale $o(t^{-d/2})$), we must show that $\prob_{\bm{0}}(\wt{D}_{t} = \bm{0})$ is also asymptotic to $\frac{1}{(2\uppi t)^{d/2}}$. Results with this level of precision generally require analytic techniques to prove, and indeed this is the route we take. In the rest of this section, we translate both walks $(\wt{D}_{t})$ and $(D_{t})$ into discrete time and prove some basic results in this area to prepare for the generating function analysis in Section \ref{gen_func}.

\subsection{Dynamics of the difference walks}\label{dynamics}

We now derive the transition rates for $(\wt{D}_{t})$ and $(D_{t})$, beginning with the former. For distinct $x,y \in \Z^{d}$, define  

\begin{equation}\label{difference_rates}
	\wt{Q}(x,y) = \begin{cases} 
		\frac{1}{2d} \quad&|x-y| = 1 \text{ and not both }x,y \in \mathcal{B}_{d}(1) \\ 
		\frac{1}{4d} \quad& x = \bm{0} \text{ and }y \in \mathcal{S}_{d}(1), \text{ or } x \in \mathcal{S}_{d}(1) \text{ and }y = \bm{0} \\
		\frac{1}{8d}\quad&x \in \mathcal{S}_{d}(1)\text{ and }y = -x \\ 
		0 \quad&\text{otherwise}
	\end{cases}
\end{equation}

One can verify using (\ref{coupled_rates}) that the walk $(\wt{D}_{t})$ has transition rates given by $x \xrightarrow{\wt{Q}(x,y)} y$. We remark that these are the dynamics of a continuous-time simple random walk on $\Z^{d}$ with a perturbation inside the unit ball $\mathcal{B}_{d}(1)$.

If we define the diagonal entries $\wt{Q}(x,x) =-\sum_{y: y \neq x}\wt{Q}(x,y)$, then $\wt{Q}$ is the rate matrix for $(\wt{D}_{t})$. Note that $\max_{x}\wt{Q}(x,x) = 1$, and this maximum is achieved by every $x \not\in \mathcal{B}_{d}(1)$. It is standard (see \cite[Chapter 20]{levin2017markov}) that $\wt{P} = I + \wt{Q}$ is the transition matrix for a discrete-time random walk $(\wt{X}_{n})$ on $\Z^{d}$ with the following property: with $N_{t} \sim \text{Po}(t)$, for any $x,y \in \Z^{d}$ and $t \geq 0$ we have

\begin{equation}\label{poisson_sample}
	\begin{aligned}
		\prob_{x}(\wt{D}_{t} = y) =& \sum_{n=0}^{\infty}\prob(N_{t} = n) \cdot \prob_{x}(\wt{X}_{n} = y)\\
		=&e^{-t}\sum_{n=0}^{\infty}\frac{t^{n}}{n!} \cdot \wt{P}^{n}(x,y).
	\end{aligned}
\end{equation}

Using (\ref{difference_rates}) and the fact that $\wt{P} = I + \wt{Q}$, the transition probabilities for the walk $(\wt{X}_{n})$ are 

\begin{equation}\label{difference_probs}
	\wt{P}(x,y) = \begin{cases} 
		\frac{1}{2d} \quad&|x-y| = 1 \text{ and not both }x,y \in \mathcal{B}_{d}(1) \\ 
		\frac{1}{4d} \quad& x = \bm{0} \text{ and }y \in \mathcal{S}_{d}(1), \text{ or } x \in \mathcal{S}_{d}(1) \text{ and }y = \bm{0} \\
		\frac{1}{8d}\quad&x \in \mathcal{S}_{d}(1)\text{ and }y = -x \\ 
		\frac{1}{8d} \quad& x = y \in \mathcal{S}_{d}(1) \\
		\frac{1}{2} \quad&x = y = \bm{0}\\
		0 \quad& \text{otherwise}
	\end{cases}
\end{equation}

We note in particular that $\wt{P}$ is a symmetric matrix.

Turning to the independent setting, $(D_{t})$ has rate matrix given by 

$$
	Q(x,y) = \begin{cases} 
		\frac{1}{2d}\quad&|x-y| = 1\\
		-1\quad& x=y\\
		0\quad&\text{ otherwise}
	\end{cases}
$$

The corresponding discrete-time walk $(X_{n})$ has transition matrix $P = I + Q$;  $(X_{n})$ is a simple random walk on $\Z^{d}$, and note that the correspondence (\ref{poisson_sample}) also holds between $(D_{t})$ and $(X_{n}).$ 

We will use the following shorthand for the rest of the paper: 

\begin{equation}\label{sequences}
	p_{n} = \prob_{\bm{0}}(X_{n} = \bm{0}) = P^{n}(\bm{0}, \bm{0}) \quad \text{ and } \quad \wt{p}_{n} = \prob_{\bm{0}}(\wt{X}_{n} = \bm{0}) = \wt{P}^{n}(\bm{0}, \bm{0}).
\end{equation}

\subsection{Properties of the discrete-time walks}

We conclude this section by establishing some properties of the walks $(X_{n})$ and $(\wt{X}_{n})$. For background on random walks on infinite graphs, we refer the reader to \cite{woess2000random}, which we use as a source for basic definitions and properties. 

Let $(S_{n})$ be a random walk on a discrete set $\mathcal{X}$. We assume that $(S_{n})$ is \textit{irreducible}, that is, for any pair $x,y \in \mathcal{X}$, $\prob_{x}(S_{n} = y) > 0$ for some $n$. The \textit{period} of $(S_{n})$ is the greatest common divisor of the set $\{\prob_{x}(S_{n} = x)\}$; the \textit{strong period} of $(S_{n})$ is the greatest common divisor of $\{n\,:\,\inf_{x} \{\prob_{x}(S_{n}=x)\}>0\}$, whenever this set is nonempty. It is straightforward to show that these definitions do not depend on $x$ when $(S_{n})$ is irreducible.

The \textit{spectral radius} of a random walk $(S_{n})$ is the quantity

$$
	\limsup_{n \to \infty}\prob_{x}(S_{n} = y)^{1/n} \in (0,1].
$$

This definition is independent of the choice of $x$ and $y$, owing again to irreducibility. Equivalently, it is the spectral radius (in the operator theoretic sense) of the transition matrix of $(S_{n})$ considered as a linear operator on the space $\ell^{2}(X)$.

The simple random walk $(X_{n})$ on $\Z^{d}$ has period and strong period $2$. The modified walk $(\wt{X}_{n})$ has period $1$ due to the lazy steps at points in $\mathcal{B}_{d}(1)$, but strong period $2$ since $\inf_{x}\{\prob_{x}(\wt{X}_{n}=x)\}$ is positive precisely when $n$ is even. We compute the spectral radii in the next proposition. 

\begin{proposition}\label{spectral_radius}
	Let $\uprho$ and $\wt{\uprho}$ be the spectral radii of $(X_{n})$ and $(\wt{X}_{n})$, respectively. Then $\uprho = \wt{\uprho} = 1$.
\end{proposition}

\begin{proof}
That $\uprho = 1$ is a standard result, and follows from the definition of spectral radius and the fact that $P^{n}(\bm{0}, \bm{0}) = \Uptheta(n^{-d/2})$. To show $\wt{\uprho} = 1$, we first show that $(\wt{X}_{n})$ has a \textit{sub-linear rate of escape}, i.e., $\frac{|\wt{X}_{n}|}{n} \xrightarrow{\prob} 0$ when $(\wt{X}_{n})$ is started from the origin.

Let $\wt{X}_{0} = \bm{0}$, and set $\wt{R}_{n} = \norm{\wt{X}_{n} }^{2}$ for $n \geq 0$ and $\wt{\uptheta}_{n} = \wt{X}_{n} - \wt{X}_{n-1}$ for $n \geq 1$. Then for any $x \in \Z^{d}$, $n \geq 0$,

\begin{align*}
	\E\left[\wt{R}_{n+1} - \wt{R}_{n}\,|\, \wt{X}_{n} = x\right] =& \E\left[\langle\wt{X}_{n} + \wt{\uptheta}_{n+1}, \wt{X}_{n} + \wt{\uptheta}_{n+1} \rangle - \langle \wt{X}_{n}, \wt{X}_{n} \rangle \,|\, \wt{X}_{n} = x \right] \\
	=&\E\left[2\langle \wt{X}_{n}, \wt{\uptheta}_{n+1} \rangle + \langle \wt{\uptheta}_{n+1}, \wt{\uptheta}_{n+1} \rangle \,|\,\wt{X}_{n} = x \right]
\end{align*}

The term $\E[2\langle \wt{X}_{n},\wt{\uptheta}_{n+1}\rangle\,|\,\wt{X}_{n} = x]$ is $0$ for any $x$ because the increments of the walk are symmetric. (This is obvious from the transition probabilities (\ref{difference_probs}) everywhere except perhaps for $x \in \mathcal{S}_{d}(1)$, which can be easily checked.) The term $\E[\langle \wt{\uptheta}_{n+1}, \wt{\uptheta}_{n+1} \rangle \,|\,\wt{X}_{n} = x]$ is the squared length of the increment vector, which is, crudely, at most $4$ for any $x$ (this simply accounts for the ``reflection" steps of length $2$ in $\mathcal{S}_{d}(1)).$ Hence $\E[\wt{R}_{n+1} - \wt{R}_{n}] \leq 4$ and, consequently, $\E[\wt{R}_{n}] \leq 4n.$ Then
$$
	\E[|\wt{X}_{n}|] \leq \E\left[ \sqrt{d\wt{R}_{n}} \right] \leq  \sqrt{4d n}
$$ 

where the first inequality follows from an application of Cauchy-Schwarz, and the second from Jensen's inequality. This is enough to give $\frac{|\wt{X}_{n}|}{n} \xrightarrow{\prob} 0$ as desired.

To finish, by \cite[Lemma 1.9]{woess2000random} we have that $\wt{P}^{n}(x, x) \leq \wt{\uprho}^{n}$ for all $x$, and thus by the symmetry of $\wt{P}$

$$
	(\wt{P}^{n}(x,y))^{2} = \wt{P}^{n}(x,y)\wt{P}^{n}(y,x) \leq \wt{P}^{2n}(x,x) \leq \wt{\uprho}^{2n}.
$$

So 

\begin{equation}\label{spectral_bound}
	\prob(|\wt{X}_{n} | \leq n) \leq \text{Vol}(\mathcal{B}_{d}(n))\wt{\uprho}^{n} = O(n^{d}\wt{\uprho}^{n}).
\end{equation}

On the other hand, $\frac{|\wt{X}_{n} |}{n} \xrightarrow{\prob} 0$ implies that $\prob(|\wt{X}_{n}| >  n) \to 0$ as $n \to \infty$. But this contradicts (\ref{spectral_bound}) unless $\wt{\uprho} = 1$. 

\end{proof}

\section{Generating functions and singularity analysis}\label{gen_func}

\subsection{Green's functions}

Recall from Section \ref{discrete_rw} the definitions of the walks $(X_{n})$ and $(\wt{X}_{n})$ on $\Z^{d}$.  Let $p_{n}$ and $\wt{p}_{n}$ be as in (\ref{sequences}). Formally define the \textit{Green's functions} 

$$
	G(z) = \sum_{n=0}^{\infty}p_{n}z^{n} \quad \text{ and } \quad \wt{G}(z) = \sum_{n=0}^{\infty}\wt{p}_{n}z^{n}.
$$ 

Observe that $G(z)$ is an even function, since all of its odd-power coefficients are $0$.

Our goal for this section is to derive an asymptotic estimate for the sequence $(\wt{p}_{n})$ by performing singularity analysis on $\wt{G}(z)$. We achieve this by establishing a relationship between $\wt{G}(z)$ and $G(z)$ in Proposition \ref{gf_prop}. The proof requires some auxiliary generating functions, which we define now.

For a subset $S \subseteq \Z^{d}$, let $\uptau_{S} =\inf\{n \geq 1: X_{n} \in S\}$ and $\wt{\uptau}_{S} = \inf\{n \geq 1: \wt{X}_{n} \in S\}$. When $S$ is a singleton $\{x\}$, we simply write $\uptau_{x}$ and $\wt{\uptau}_{x}$. For $n \geq 1$, let 

$$
q_{n}=\prob_{\bm{0}}(\uptau_{\bm{0}} = n) \quad \text{ and } \quad \wt{q}_{n}=\prob_{\bm{0}}(\wt{\uptau}_{\bm{0}} = n)
$$

and formally define $Q(z) = \sum_{n=1}^{\infty}q_{n}z^{n}$ and $\wt{Q}(z) = \sum_{n=1}^{\infty}\wt{q}_{n}z^{n}.$ For $n \geq 0$ define

$$
r_{n} = \prob_{\mathcal{S}_{d}(1)}(X_{n} \in \mathcal{S}_{d}(1),\,\uptau_{\bm{0}} > n) \quad \text{ and } \quad \wt{r}_{n} = \prob_{\mathcal{S}_{d}(1)}(\wt{X}_{n} \in \mathcal{S}_{d}(1),\,\wt{\uptau}_{\bm{0}} > n)
$$

and let $R(z), \wt{R}(z)$ be the corresponding generating functions. By conditioning on the events $\{X_{0} \in \mathcal{S}_{d}(1)\}$ and $\{\wt{X}_{0} \in \mathcal{S}_{d}(1)\}$, we mean conditioning on $X_{0}$, $\wt{X}_{0}$ being at some given $x \in \mathcal{S}_{d}(1)$; by the symmetry of the walks, the definitions of $r_{n}$ and $\wt{r}_{n}$ do not depend on the choice of $x$. We keep this abuse of notation for convenience.

Finally, for $n \geq 1$, let

$$
s_{n}=\prob_{\mathcal{S}_{d}(1)}(\uptau_{\mathcal{S}_{d}(1)} = n,\,\uptau_{\bm{0}} > n) \quad \text{ and } \quad \wt{s}_{n}:=\prob_{\mathcal{S}_{d}(1)}(\wt{\uptau}_{\mathcal{S}_{d}(1)} = n,\,\wt{\uptau}_{\bm{0}} > n)
$$

with corresponding generating functions $S(z)$ and $\wt{S}(z)$.

\begin{proposition}\label{gf_prop}
	The Green's functions $G(z)$ and $\wt{G}(z)$ satisfy the formal relation
		\begin{equation}\label{gf0}
			\wt{G}(z) = \frac{1 - (1-2z)G(z)}{1 - (1-z)^{2}G(z)}
		\end{equation}
\end{proposition}

\begin{proof}

We rely throughout on the transition probabilities for the walk $(\wt{X}_{n})$ given in (\ref{difference_probs}). For any $n \geq 1$ we have

\begin{align*}
\wt{p}_{n} =&\sum_{j=1}^{n}\prob_{\bm{0}}(\wt{X}_{n} = \bm{0}\,|\,\wt{\uptau}_{\bm{0}} = j)\prob_{\bm{0}}(\wt{\uptau}_{\bm{0}}=j)\\
=&\sum_{j=1}^{n}\prob_{\bm{0}}(\wt{X}_{n-j} = \bm{0})\prob_{\bm{0}}(\wt{\uptau}_{\bm{0}}=j)\\
=&\sum_{j=1}^{n}\wt{p}_{n-j}\wt{q}_{j}
\end{align*}

and $\wt{p}_{0} = 1$. Thus we have the functional identity 

\begin{equation}\label{gf1}
	\wt{G}(z) = 1 + \wt{G}(z)\wt{Q}(z)
\end{equation}

Observe that

$$\wt{q}_{1} = \prob_{\bm{0}}(\wt{\uptau}_{\bm{0}} = 1) = \prob_{\bm{0}}(\wt{X}_{1} = \bm{0}) = \frac{1}{2}.$$

For $n \geq 2$, if $\wt{X}_{0} = \bm{0}$ and $\wt{\uptau}_{\bm{0}} = n$, the trajectory $(\wt{X}_{0}, \wt{X}_{1}, \dots, \wt{X}_{n})$ consists precisely of an initial step from $\bm{0}$ to $\mathcal{S}_{d}(1)$, a walk of length $n-2$ which begins and ends in $\mathcal{S}_{d}(1)$ and avoids $\bm{0}$, and a final return step from $\mathcal{S}_{d}(1)$ to $\bm{0}.$ Thus

$$ \wt{q}_{n} = \prob_{\bm{0}}(\wt{X}_{1} \in \mathcal{S}_{d}(1)) \cdot \hspace{9cm}$$
$$\prob_{\mathcal{S}_{d}(1)}(\wt{X}_{n-2} \in \mathcal{S}_{d}(1),\,\wt{\uptau}_{\bm{0}} > n-2) \cdot$$
$$\hspace{8cm}\prob_{\mathcal{S}_{d}(1)}(\wt{X}_{1} = \bm{0})$$

which is equivalent to $\wt{q}_{n} = \frac{1}{8d}\wt{r}_{n-2}.$ Thus

\begin{equation}\label{gf2}
	\wt{Q}(z) = \frac{z}{2} + \sum_{n=2}^{\infty}\wt{q}_{n}z^{n} = \frac{z}{2} + \frac{z^{2}\wt{R}(z)}{8d}.
\end{equation}

By conditioning on possible values of $\wt{\uptau}_{\mathcal{S}_{d}(1)}$ as in the derivation of (\ref{gf1}), one can derive

\begin{equation}\label{gf3}
\wt{R}(z) = 1 + \wt{R}(z)\wt{S}(z).
\end{equation}

Combining (\ref{gf1}), (\ref{gf2}), and (\ref{gf3}) gives

\begin{equation}\label{gf4}
\wt{G}(z) = \frac{1 - \wt{S}(z)}{(1-z/2)(1-\wt{S}(z)) - z^{2}/8d}
\end{equation}

Now, we focus on $\wt{S}(z)$. When $\wt{X}_{0} \in \mathcal{S}_{d}(1)$, the only ways for $\wt{\uptau}_{\mathcal{S}_{d}(1)}$ to equal $1$ are for the first step in the walk to be a lazy step ($\wt{X}_{1} = \wt{X}_{0}$) or a reflection ($\wt{X}_{1} = -\wt{X}_{0}$). One of these occurs with probability $\frac{1}{8d} + \frac{1}{8d} = \frac{1}{4d}$, and hence $\wt{s}_{1} = \frac{1}{4d}$. 

We claim that for $n \geq 2$, $\wt{s}_{n} = s_{n}$. To see this, note that the hitting times $\uptau_{\mathcal{S}_{d}(1)}$ and $\wt{\uptau}_{\mathcal{S}_{d}(1)}$ have the same distribution whenever $X_{0} = \wt{X}_{0} \not\in \mathcal{B}_{d}(1)$, due to the fact that the walks $(X_{n})$ and $(\wt{X}_{n})$ have identical transition probabilities outside of $\mathcal{B}_{d}(1)$. Moreover, for any $x \in \mathcal{S}_{d}(1)$ and $y \in \mathcal{S}_{d}(2)$, we have $\prob_{x}(\wt{X}_{1} = y) = \prob_{x}(X_{1} = y)$. Thus, for any $n \geq 2$ and $x \in \mathcal{S}_{d}(1)$,

\begin{align*}
\prob_{x}(\wt{\uptau}_{\mathcal{S}_{d}(1)} = n,\,\wt{\uptau}_{\bm{0}} > n) =& \sum_{y \in \mathcal{S}_{d}(2)}\prob_{y}(\wt{\uptau}_{\mathcal{S}_{d}(1)} = n-1)\prob_{x}(\wt{X}_{1} = y)\\
 =& \sum_{y \in \mathcal{S}_{d}(2)}\prob_{y}(\uptau_{\mathcal{S}_{d}(1)} = n-1)\prob_{x}(X_{1} = y)\\
 =&\prob_{x}(\uptau_{\mathcal{S}_{d}(1)} = n,\,\uptau_{\bm{0}} > n)
\end{align*}

Since $s_{1} = 0$, we have 

\begin{equation}\label{gf5}
\wt{S}(z) = \frac{z}{4d} + S(z).
\end{equation}

In the same manner in which we derived (\ref{gf1}), (\ref{gf2}), and (\ref{gf3}) for the walk $(\wt{X}_{n})$, one can derive for $(X_{n})$ that

\begin{align*}
G(z) =& 1 + G(z)Q(z)\\
Q(z) =& \frac{z^{2}R(z)}{2d}\\
R(z) =& 1 + R(z)S(z)
\end{align*}

from which we can solve for $S(z)$ to reach

\begin{equation}\label{gf6}
S(z) = 1 - \frac{z^{2}}{2d}\cdot\frac{G(z)}{G(z) -1}
\end{equation}

From (\ref{gf5}) and (\ref{gf6}), we get $\wt{S}(z) = 1 + \frac{z}{4d} - \frac{z^{2}}{2d}\cdot\frac{G(z)}{G(z)-1}$; substituting this into (\ref{gf4}) and performing some simplifications yields (\ref{gf0}).

\end{proof}

\subsection{Background on singularity analysis}

Our framework for the remainder of the section is singularity analysis in the manner of Flajolet \& Sedgwick \cite{flajolet2009analytic}, from which we now recall some notation and conventions. For $R > 1$ and $0 < \upphi < \frac{\uppi}{2}$, let

$$
	\Updelta(\upphi, R) := \{z \in \C\,:\,|z| < R,\,|\arg(z-1)| > \upphi\}.
$$

Any set of this form is called $\Updelta$-\textit{domain}. For $a,b \in \C$, define 

\begin{equation}\label{standard_scale}
	\bm{h}_{a,b}(z):=(1-z)^{-a}\log^{b}\left(\frac{1}{1-z}\right).
\end{equation}

We call the set $\{\bm{h}_{a,b}(z)\,:\,a,b \in \C\}$ the \textit{standard scale} of functions singular at $z = 1$. For any $a,b \in \C$, the function $\bm{h}_{a,b}(z)$ is analytic on some $\Updelta$-domain. We recall three results from \cite[Section VI.2]{flajolet2009analytic} which will be of use to us. They are asymptotic expressions for the coefficient $[z^{n}]\bm{h}_{a,b}(z)$ of the power series for $\bm{h}_{a,b}(z)$ centered at $z = 0$ for certain values of $a$ and $b$. First, when $a \not\in \Z_{\leq 0}$, we have

\begin{equation}\label{standard_scale_asymp1}
	[z^{n}]\bm{h}_{a,0}(z) = \frac{n^{a-1}}{\Upgamma(a)}\left(1 + O\left( \frac{1}{n} \right) \right).
\end{equation}

Second, when $a \in \Z_{\leq 0}$,

\begin{equation}\label{standard_scale_asymp2}
	[z^{n}]\bm{h}_{a,1}(z) = (-1)^{a}\Upgamma(-a)n^{a-1}\left(1 + O\left(\frac{1}{n} \right) \right).
\end{equation}

Finally,

\begin{equation}\label{standard_scale_asymp3}
	[z^{n}]\bm{h}_{0,-1}(z) = O\left(\frac{1}{n\log^{2}n }\right)
\end{equation}

(For (\ref{standard_scale_asymp2}) and (\ref{standard_scale_asymp3}) in particular, see \cite[p.\ 387]{flajolet2009analytic}.) To perform singularity analysis on $\wt{G}(z)$, we follow a series of steps. First, we compute the radius of convergence of $\wt{G}(z)$ and locate all singularities on the circle of convergence (the \textit{dominant} singularities). Next, at each dominant singularity $\upzeta$ we obtain a \textit{singular expansion} for $\wt{G}(z)$ of the form 

$$
	\wt{G}(z) = \wt{f}(z/\upzeta) + \bm{\upsigma}(z/\upzeta) + O(\bm{\uptau}(z/\upzeta)) \quad \text{ as } z \to \upzeta
$$ 

where $\bm{\upsigma}(z)$ and $\bm{\uptau}(z)$ are linear combinations of functions from the standard scale (\ref{standard_scale}), and $\wt{f}(z)$ is some function analytic at $z = 1$. Singular expansions in hand, we finally apply the general result \cite[Theorem VI.5]{flajolet2009analytic} to obtain asymptotic expressions for the coefficients $\wt{p}_{n} = [z^{n}]\wt{G}(z)$. For a more detailed treatment of this process, we invite the reader to consult \cite[Chapter VI]{flajolet2009analytic}.

All of the steps outlined above are made possible by the functional relationship (\ref{gf0}) between $\wt{G}(z)$ and $G(z)$, and the fact that the simple random walk on $\Z^{d}$ is very well understood. In particular, we obtain singular expansions of $\wt{G}(z)$ in straightforward manner from singular expansions of $G(z)$. 

It is worth remarking that understanding the precise behavior of $G(z)$ near its singularities is a fairly delicate procedure (except when $d=1$, in which case we have the explicit formula $G(z) = \frac{1}{\sqrt{1-z^{2}}}$). While it is well-known that $p_{n} = [z^{n}]G(z) = (1+o(1))\frac{2}{(2\uppi n)^{d/2}}$ for even $n$, one cannot deduce singular expansions for $G(z)$ from this information alone by, say, matching the asymptotics of $p_{n}$ to the asymptotics of $[z^{n}]\bm{h}_{a,b}(z)$ in (\ref{standard_scale_asymp1})-(\ref{standard_scale_asymp3})---to do this, one must know \textit{a priori} that $G(z)$ admits singular expansions of this form at its dominant singularities in the first place. Most of the heavy lifting is done for us in a result from \cite{woess2000random}, which we are able to cite with only small modifications.

\subsection{Singularity analysis of the Green's functions}

We proceed with singularity analysis, beginning with the location of the dominant singularities.

\begin{proposition}\label{dominant_singularities}
	The Green's functions $G(z)$ and $\wt{G}(z)$ have radii of convergence $1$ and dominant singularities at $z = \pm 1$
\end{proposition}

\begin{proof}

 \cite[Theorem 9.4]{woess2000random} gives that the Green's function of any irreducible random walk with spectral radius $1$ has dominant singularities contained in $\{e^{2\uppi i k / d_{s}}\}_{k=0}^{d_{s}}$, where $d_{s}$ is the strong period, and further that these points comprise all dominant singularities when the period and strong period coincide. Moreover, since the coefficients of a Green's function are nonnegative and real, $z = 1$ must be a singularity by Pringsheim's theorem.
 
By Proposition \ref{spectral_radius}, both $(X_{n})$ and $(\wt{X}_{n})$ have spectral radius $1$. Since $(X_{n})$ has period and strong period $2$, $G(z)$ has dominant singularities at precisely $z = \pm 1$. The dominant singularities of $\wt{G}(z)$ are contained in $\{-1,1\}$, since $(\wt{X}_{n})$ has period $1$ but strong period $2$. To finish, we show that $z = -1$ must be a singularity for $\wt{G}(z)$.

Suppose for the sake of contradiction that $\wt{G}(z)$ has an analytic continuation to some neighborhood of $z = -1$. By solving the relation (\ref{gf0}) for $G(z)$, we find 

\begin{equation}\label{gf_reversed}
	G(z) = \frac{1 - \wt{G}(z)}{(1-2z) - (1-z)^{2}\wt{G}(z)}
\end{equation}

Since $G(z)$ has a singularity at $z = -1$ and $\wt{G}(z)$ is analytic in a neighborhood of $z = -1$, it must be the case that the denominator of (\ref{gf_reversed}) has a zero at $z = -1$---else $G(z)$ is analytic there as well. Note that if the denominator of (\ref{gf_reversed}) is $0$ at $z = -1$, the numerator cannot also be $0$ there. Thus $G(z)$ has a pole of order $m \geq 1$ at $z = -1$. We may then express $G(z) = \frac{H(z)}{(1+z)^{m}}$ in a neighborhood of $z = -1$ for some analytic $H(z)$ with $H(-1) \neq 0$. Since $G(z)$ is an even function, $G(z) = \frac{H(-z)}{(1-z)^{m}}$ in a neighborhood of $z = 1$. But by basic singularity analysis, this implies the coefficient asymptotics

$$
	p_{2n} = [z^{2n}]G(z) = 2H(-1)n^{m-1}(1+o(1))
$$

which is clearly a contradiction for any value of $m \geq 1$. Therefore we conclude $\wt{G}(z)$ must be singular at $z = -1$.

\end{proof}

Proposition \ref{dominant_singularities} implies that both $G(z)$ and $\wt{G}(z)$ admit analytic continuations to a domain of the form $\Updelta_{0} \cap -\Updelta_{0}$ for some $\Updelta$-domain $\Updelta_{0}$. (For a set $A \subseteq \C$, $-A$ denotes the image of $A$ under the map $z \mapsto -z$.) Analyticity on such a domain is necessary to apply \cite[Theorem VI.5]{flajolet2009analytic}.

We now establish singular expansions of $G(z)$ at its dominant singularities. These follow from a minor modification of \cite[Proposition 17.16]{woess2000random}.

\begin{proposition}\label{srw_sing_exp}
	For any $d \geq 1$, the Green's function $G(z)$ has a singular expansion in a neighborhood of $z=1$ given by
		\begin{equation}\label{g_sing_exp1}
			G(z) = \begin{cases}
				f( z) + \upbeta_{d}(1 - z)^{\frac{d-2}{2}} + O\left((1 - z)^{\frac{d}{2}} \right) \quad& d \text{ odd}\\
				f( z) + \upbeta_{d}(1 - z)^{\frac{d-2}{2}}log\left(\frac{1}{1 - z} \right) + O\left((1- z)^{\frac{d}{2}}\log\left(\frac{1}{1 - z}\right) \right) \quad& d \text{ even}
			\end{cases}
		\end{equation}
	where $f(z)$ is analytic at $z =1$ and depends on $d$, and 
	
		$$
			\upbeta_{d} = 
				\begin{cases}
					\frac{\Upgamma\left(-\frac{d-2}{2}\right)}{(2\uppi)^{d/2}} \quad& d \text{ odd}\\
					\frac{(-1)^{\frac{d-2}{2}}}{\Upgamma\left(\frac{d-2}{2}\right)(2\uppi)^{d/2}} \quad& d \text{ even}
				\end{cases}
		$$
\end{proposition}

In terms of the standard scale (\ref{standard_scale}), we may express (\ref{g_sing_exp1}) more compactly as 

$$
	G(z) = f(z) + \upbeta_{d}\bm{h}_{-\frac{d-2}{2},b}(z) + O\left(\bm{h}_{-\frac{d}{2},b}(z) \right)
$$

where $b = (d+1)\mod 2$. Since $G(z)$ is an even function, we also get the singular expansion $G(z) = f(-z) + \upbeta_{d}\bm{h}_{-\frac{d-2}{2},b}(-z) + O\left(\bm{h}_{-\frac{d}{2},b}(-z) \right)$ in a neighborhood of $z = -1$.

\begin{proof}
The cited result from \cite{woess2000random} gives that $G(z)$ has a singular expansion in a neighborhood of $z=1$ of the form $f(z) + g(z)\bm{h}_{-\frac{d-2}{2},b}(z)$, where $f(z)$ and $g(z)$ are analytic functions which depend on $d$, and $g(1) \neq 0$. By expanding $g(z)$ as a power series $\sum_{n=0}^{\infty}g_{n}(1-z)^{n}$ at $z = 1$, we see that $g(z) = g_{0} + O(1-z)$ as $z \to 1$ with $g_{0} \neq 0$, so 

$$
	G(z) = f(z) + g_{0}\bm{h}_{-\frac{d-2}{2},b}(z) + O\left(\bm{h}_{-\frac{d}{2},b}(z)\right).
$$

The symmetric expansion at $z = - 1$ is given by 
$$
	G(z) = f(-z) + g_{0}\bm{h}_{-\frac{d-2}{2}, b}(-z) + O\left(\bm{h}_{-\frac{d}{2},b}(-z)\right).
$$

To complete the proof we verify that $g_{0} = \upbeta_{d}.$ The local limit theorem for the simple random walk on $\Z^{d}$ (see, for instance, \cite[Theorem 2.1.3]{lawler2010random}) gives $p_{2n} = [z^{2n}]G(z) =(1+o(1)) \frac{2}{(4\uppi n)^{d/2}}$. On the other hand, using (\ref{standard_scale_asymp1}) and (\ref{standard_scale_asymp2}) along with \cite[Theorem VI.5]{flajolet2009analytic}, singularity analysis of the expansions for $G(z)$ given above implies $p_{2n} = (1+o(1)) \frac{C_{d}}{(2n)^{d/2}}$ where

$$
	C_{d} = 
		\begin{cases}
			\frac{2g_{0}}{\Upgamma\left(-\frac{d-2}{2} \right)} \quad& d \text{ odd}\\
			(-1)^{\frac{d-2}{2}}\cdot 2g_{0}\Upgamma\left(\frac{d-2}{2}\right) \quad& d \text{ even}
		\end{cases}
$$

We must then have $\frac{C_{d}}{2^{d/2}} = \frac{2}{(4\uppi)^{d/2}}$. Solving this for $g_{0}$ gives 

$$
	g_{0} = 
		\begin{cases}
			\frac{\Upgamma\left(-\frac{d-2}{2}\right)}{(2\uppi)^{d/2}} \quad& d \text{ odd}\\
			\frac{(-1)^{\frac{d-2}{2}}}{\Upgamma\left(\frac{d-2}{2}\right)(2\uppi)^{d/2}} \quad& d \text{ even}
		\end{cases}
$$

That is, $g_{0} = \upbeta_{d}$.

\end{proof}

Using Proposition \ref{srw_sing_exp} along with the relation (\ref{gf0}), we can obtain singular expansions for $\wt{G}(z)$ at its dominant singularities. 

\begin{proposition}\label{prw_sing_exp}

	For all $d \geq 1$, the Green's function $\wt{G}(z)$ has a singular expansion in a neighborhood of $z = 1$ of the form

	$$
		\wt{G}(z) = \wt{f}_{+}(z) + \upbeta_{d}\bm{h}_{-\frac{d-2}{2},b}(z) + O\left(\bm{h}_{-\frac{d}{2},b}(z) \right)
	$$
	
	where $\wt{f}_{+}(z)$ is analytic at $z = 1$ and depends on $d$, and $b=(d+1)\mod 2$.
	
	\vspace{0.25cm}
	
	For $d = 1,2$, in a neighborhood of $z = -1$ we have
	
	$$
		\wt{G}(z) = \wt{f}_{-}(-z) + O\left(\bm{h}_{\frac{d-2}{2},-b}(-z) \right),
	$$
	
	and for $d \geq 3$, 
	
	$$
		\wt{G}(z) = \wt{f}_{-}(-z) + \frac{\upbeta_{d}}{(4\upalpha_{d}-1)^{2}}\bm{h}_{-\frac{d-2}{2},b}(-z) + O\left(\bm{h}_{-\frac{d}{2},b}(-z) \right)
	$$
	
	where $\wt{f}_{-}(z)$ is analytic at $z = 1$ and depends on $d$, and $\upalpha_{d}$ is the expected number visits of the simple random walk  $(X_{n})$ to the origin when $X_{0} = \bm{0}$.

\end{proposition}

We see from Propositions \ref{srw_sing_exp} and \ref{prw_sing_exp} that $G(z)$ and $\wt{G}(z)$ have nearly identical expansions at the singularity $z = 1$. It is only at $z = -1$ that we observe differing behaviors. 

\begin{proof}
	Throughout the proof, we use $a_{1}(z), a_{2}(z), a_{3}(z),\dots$ to denote ``utility functions" which condense some computations. Such a function $a_{j}(z)$ will always be (manifestly) analytic at $z = 1$. 
	
	We begin with the expansion at $z =1$. With $f(z)$ as in (\ref{g_sing_exp1}), define $s(z) = G(z) - f(z)$ in a neighborhood of $z = 1$. We also define $a_{1}(z) =1- (1-z)^{2}f(z)$. Note that $a_{1}(1) = 1$, so we may assume that we work in a small enough neighborhood of $z=1$ so that $a_{1}(z) \neq 0$. 
	
	From (\ref{gf0}),
	
	\begin{align*}
		\wt{G}(z)  =& \frac{1-(1-2z)G(z)}{1-(1-z)^{2}G(z)} \\
		=& 1 + G(z)\left(\frac{z^{2}}{1 - (1-z)^{2}G(z)}\right)\\
		=&1 + G(z)\left(\frac{z^{2}}{a_{1}(z) - (1-z)^{2}s(z)} \right)\\
		=&1 + G(z)\left(\frac{z^{2}}{a_{1}(z)} \cdot \frac{1}{1 - (1-z)^{2}s(z)/a_{1}(z)}\right)
	\end{align*}

	Note that $(1-z)^{2}s(z) = O(\bm{h}_{-\frac{d+2}{2},b}(z)) \to 0$ as $z \to 1$ for any $d \geq 1$. Hence the ratio $|(1-z)^{2}s(z)| / |a_{1}(z)|$ is less than $1$ for $z$ sufficiently near $1$, and we may write
	
	\begin{equation}\label{sing_exp1}
		\frac{z^{2}}{a_{1}(z)} \cdot \frac{1}{1 - (1-z)^{2}s(z)/a_{1}(z)} =\frac{z^{2}}{a_{1}(z)}\left(1 + O\left(\bm{h}_{-\frac{d+2}{2},b}(z) \right)\right)
	\end{equation}

	Since $a_{1}(z)$ is analytic and $a_{1}(z) = 1$, we can write $\frac{z^{2}}{a_{1}(z)} = 1 + a_{2}(z)$ where $a_{2}(z)$ is analytic and satisfies $a_{2}(z) = O(1-z)$ as $z \to 1$. Thus (\ref{sing_exp1}) is $1 + a_{2}(z) + O\left(\bm{h}_{-\frac{d+2}{2},b}(z) \right)$. Using Proposition \ref{srw_sing_exp}, we get the singular expansion
	
		\begin{align*}
			\wt{G}(z) =& 1 +G(z)\left(1 + a_{2}(z) + O\left(\bm{h}_{-\frac{d+2}{2},b}(z) \right)\right)\\
			=&\wt{f}_{+}(z) + \upbeta_{d}\bm{h}_{-\frac{d-2}{2},b}(z) + O\left(\bm{h}_{-\frac{d}{2},b}(z)\right)
		\end{align*}
		
	where $\wt{f}_{+}(z) = 1 + f(z)\left(1 + a_{2}(z)\right)$ is analytic at $z = 1$ and depends on $d$.
	
	Now we move to the expansion at $z = -1$. Here we must consider the recurrent ($d \leq 2$) and transient ($d \geq 3$) cases separately. We begin with the former. Since $G(z)$ is an even function, to obtain a singular expansion of $\wt{G}(z) = \frac{1-(1-2z)G(z)}{1-(1-z)^{2}G(z)}$ at $z = -1$, we may instead find a singular expansion for $\widehat{G}(z) = \wt{G}(-z) = \frac{1 - (1+2z)G(z)}{1 - (1+z)^{2}G(z)}$ at $z = 1$. 
	
	Let $a_{3}(z) = 1-(1+2z)f(z)$, $a_{4}(z) = 1 - (1+z)^{2}f(z)$, and $a_{5}(z) = \frac{a_{4}(z)}{(1+z)^{2}}$. Since for $d \leq 2$ we have $s(z) \to \infty$ as $z \to 1$, we may take a small enough neighborhood of $z = 1$ so that $s(z) \neq 0$. Then
	
	\begin{align*}
		\widehat{G}(z) =& \frac{1-(1+2z)(f(z) + s(z))}{1 - (1+z)^{2}(f(z) + s(z))} \\
		=& \frac{a_{3}(z) - (1+2z)s(z)}{a_{4}(z) - (1+z)^{2}s(z)}\\
		=& \frac{1+2z - \frac{a_{3}(z)}{s(z)}}{(1+z)^{2}} \cdot \frac{1}{1 - \frac{a_{5}(z)}{s(z)}}.
	\end{align*}
	
	We can expand $\frac{1}{1-a_{5}(z)/s(z)} = 1 + O(s(z)^{-1})$, and thus in all we get 
	
	$$
		\widehat{G}(z) = \frac{1+2z}{(1+z)^{2}} + O(s(z)^{-1})
	$$
	
	near $z = 1$. We have $s(z) = \Uptheta(\bm{h}_{-\frac{d-2}{2},b}(z))$, so $s(z)^{-1} = O(\bm{h}_{\frac{d-2}{2},-b}(z))$. This completes the $d \leq 2$ case.
	
	Now let $d \geq 3$. In this case, $G(z)$ is defined at $z = 1$, where we have $G(1) = f(1) = \upalpha_{d}$, with $\upalpha_{d}$ the expected number of visits of the walk $(X_{n})$ to the origin when $X_{0} = \bm{0}$. Clearly we have $\upalpha_{d} \geq 1$, and therefore $a_{4}(1) = 1-4\upalpha_{d} < 0$. We require this observation in order to divide by $a_{4}(z)$ in a neighborhood of $z=1$ in the following computation:
	
	\begin{align*}
		\widehat{G}(z) =& \frac{a_{3}(z) - (1+2z)s(z)}{a_{4}(z) - (1+z)^{2}s(z)}\\
		=& \frac{a_{3}(z) - (1+2z)s(z)}{a_{4}(z)}\cdot\frac{1}{1-\frac{s(z)}{a_{5}(z)}}
	\end{align*}
	
	Since $s(z) \to 0$ as $z \to 1$ for $d \geq 3$, in a neighborhood of $z=1$ we have 
	
	$$
		\frac{1}{1-s(z)/a_{5}(z)} = 1 + \frac{s(z)}{a_{5}(z)} + O(s(z)^{2})
	$$
	
	$$
		\quad = 1 + \frac{(1+z)^{2}s(z)}{a_{4}(z)} + O(s(z)^{2}).
	$$
	
	This yields
	
	$$
		\widehat{G}(z) =\frac{a_{3}(z)}{a_{4}(z)} + \left(\frac{(1+z)^{2}a_{3}(z)}{a_{4}(z)^{2}} - \frac{1+2z}{a_{4}(z)} \right) s(z) + O(s(z)^{2}).
	$$
	
	To finish, we note that $\frac{(1+z)^{2}a_{3}(z)}{a_{4}(z)^{2}} - \frac{1+2z}{a_{4}(z)}$ is analytic at $z = 1$, where it equals
	
	$$
		\frac{4a_{3}(1)}{a_{4}(1)^{2}} - \frac{3}{a_{4}(1)} = \frac{4(1-3\upalpha_{d})}{(1-4\upalpha_{d})^{2}} - \frac{3}{1-4\upalpha_{d}} = \frac{1}{(4\upalpha_{d}-1)^{2}}.
	$$
	
	Expanding this function as $\frac{1}{(4\upalpha_{d} -1)^{2}} + O(1-z)$ near $z = 1$ and using the fact that $s(z) = \upbeta_{d}\bm{h}_{-\frac{d-2}{2},b}(z) + O\left(\bm{h}_{-\frac{d}{2},b}(z)\right)$, we finally get
	
	$$
		\widehat{G}(z) = \wt{f}_{-}(z) + \frac{\upbeta_{d}}{(4\upalpha_{d}-1)^{2}}\bm{h}_{-\frac{d-2}{2},b}(z) + O\left(\bm{h}_{-\frac{d}{2},b}(z)\right) 
	$$
	
	where $\wt{f}_{-}(z) = \frac{a_{3}(z)}{a_{4}(z)}.$  
	
\end{proof}

The singular expansions in Proposition \ref{prw_sing_exp}, together with \cite[Theorem VI.5]{flajolet2009analytic} and the formulae (\ref{standard_scale_asymp1})-(\ref{standard_scale_asymp3}) immediately imply the following coefficient asymptotics for $\wt{p}_{n}$.

\begin{corollary}\label{prw_asymp}
	For $d = 1$,
	
	$$
		\wt{p}_{n} = \frac{1}{(2\uppi n)^{1/2}}\left(1 + O\left( \frac{1}{n} \right) \right).
	$$
	
	For $d = 2$, 
	
	$$
		\wt{p}_{n} = \frac{1}{2\uppi n}\left(1 + O\left(\frac{1}{\log^{2}n}\right) \right)
	$$
	
	For $d \geq 3$
	$$
		\wt{p}_{n} = \left(1 + \frac{(-1)^{n}}{(4\upalpha_{d} - 1)^{2}} \right)\frac{1}{(2\uppi n)^{d/2}}\left(1 + O\left( \frac{1}{n}\right) \right)
	$$
\end{corollary}

\section{Proofs of Theorems \ref{clt} and \ref{two_norm_theorem}}\label{proofs}

In this section we prove the main theorems, starting with Theorem \ref{two_norm_theorem} and finishing with Theorem \ref{clt}, which is in effect a corollary. We find it more natural to break the proof of Theorem \ref{two_norm_theorem} into two parts. Parts (i) and (ii) of Theorem \ref{two_norm_theorem} essentially reduce to a local limit theorem for the walk $(\wt{D}_{t})$, and to prove it we simply adapt the strategy from \cite[Theorem 2.5.6]{lawler2010random}. 

\begin{proof}[Proof of Theorem \ref{two_norm_theorem} (i) and (ii)]

	By Lemma \ref{difference_lemma}, it is enough to show that both $\prob_{\bm{0}}(\wt{D}_{t} = \bm{0})$ and $\prob_{\bm{0}}(D_{t} = \bm{0})$ are both asymptotically $\frac{1}{(2\uppi t)^{d/2}}\left(1 + O\left( \updelta_{t} \right)\right)$, where $\updelta_{t} = t^{-1/2}$ when $d \neq 2$ and $\updelta_{t} = \frac{1}{\log^{2}t}$ when $d=2$.

	We have $p_{n} = \frac{2}{(2\uppi n)^{d/2}}\left(1 + O\left(\frac{1}{n}\right) \right)$ for even $n$, and $p_{n} = 0$ for odd $n$, either by the local limit theorem or singularity analysis of the expansions given in Proposition \ref{srw_sing_exp}. We also have asymptotics for $\wt{p}_{n}$ from Corollary \ref{prw_asymp}. It is easy to see that, excluding for the moment the $d=2$ case for $\wt{p}_{n}$, we may express both $p_{n}$ and $\wt{p}_{n}$ in the general form 
	
	\begin{equation}\label{gen_asymptotic}
		a_{n} = \frac{1 + (-1)^{n}c}{(2\uppi n)^{d/2}} \exp\left\{O\left( \frac{1}{n} \right) \right\}
	\end{equation} 
	
	with $c \geq 0$ constant depending on $d$. (This also holds for $d= 2$ when the error term is replaced with $\exp\left\{O\left(\frac{1}{\log^{2}n} \right) \right\}$.) By the relation (\ref{poisson_sample}), with $N_{t} \sim \text{Po}(t)$ both $\prob_{\bm{0}}(D_{t} = \bm{0})$ and $\prob_{\bm{0}}(\wt{D}_{t} = \bm{0})$ can then be expressed as a sum
	
	$$
		\sum_{n=0}^{\infty}\prob(N_{t} = n)a_{n}.
	$$
	
	We will show that the sum above is $\frac{1}{(2\uppi t)^{d/2}} + O\left(t^{-(d+1)/2}\right)$ for any sequence $(a_{n})$ with $a_{n} \in [0,1]$ for all $n$ and satisfying (\ref{gen_asymptotic}), which establishes the result when $d \neq 2$. The $d =2$ case is essentially the same, but requires a more careful handling of the error term which we defer to the end.
	
	For $t, \upepsilon > 0$, let $I_{t}^{\upepsilon} = \{n \in \Z\,:\,|n-t| \leq \upepsilon t\}$. By standard large deviations bounds, for a fixed $\upepsilon > 0$, $\prob(N_{t} \not\in I_{t}^{\upepsilon}) = O(e^{-c_{\upepsilon}t})$ as $t \to \infty$, where $c_{\upepsilon}$ is a constant which depends on $\upepsilon.$ Thus it suffices to show that there exists an $\upepsilon > 0$ so that
	
	\begin{equation}\label{epsilon_sum}
		\sum_{n \in I_{t}^{\upepsilon}}\prob(N_{t} = n)a_{n} = \frac{1}{(2\uppi t)^{d/2}} + O(t^{-(d+1)/2}).
	\end{equation}
	
	We require the following local limit theorem for the Poisson random variable $N_{t}$: if $n \in \Z$ satisfies $|n-t| \leq t/2$, then
	
	\begin{equation}\label{poisson_llt}
		\prob(N_{t} = n) = \frac{1}{\sqrt{2 \uppi  t}}e^{-\frac{(n-t)^{2}}{2t}}\exp\left\{O\left( \frac{1}{\sqrt{t}} + \frac{|n-t|^{3}}{t^{2}} \right) \right\}.
	\end{equation}

	(See \cite[Proposition 2.5.5]{lawler2010random}.) Observe that
	
	$$
		\frac{1}{n^{d/2}} = \frac{1}{(n-t + t)^{d/2}} = \frac{1}{t^{d/2}}\exp\left\{O\left( \frac{|n-t|}{t} \right)\right\}.
	$$
	
	and
	
	$$
		\frac{1}{n} = O\left(\frac{1}{t}\right)
	$$
	
	uniformly for $n \in I_{t}^{\upepsilon}$. Hence
	
	\begin{equation}\label{a_asymptotic}
		a_{n} = \frac{1+(-1)^{n}c}{(2 \uppi t)^{d/2}}\exp\left\{ O\left(\frac{|n-t|}{t} \right) \right\}.
	\end{equation}
	
	Assuming that $\upepsilon < \frac{1}{2}$ to that (\ref{poisson_llt}) may be applied, we then have for $n \in I_{t}^{\upepsilon}$: 
	
	\begin{eqnarray}
		\prob(N_{t} = n)a_{n} =\frac{1 + (-1)^{n}c}{(2 \uppi t)^{(d+1)/2}}e^{-\frac{(n-t)^{2}}{2t}}\exp\left\{O\left(\frac{1}{\sqrt{t}} + \frac{|n-t|}{t} + \frac{|n-t|^{3}}{t^{2}} \right)\right\}. \label{general_term}
	\end{eqnarray}
	
	Note that $\frac{|n-t|}{t}$ dominates $\frac{|n-t|^{3}}{t^{2}}$ until $|n-t|$ is of the order $\sqrt{t}$, at which point $\frac{|n-t|^{3}}{t^{2}}$ becomes the dominant term. Thus we may simplify the error to $\exp\left\{O\left(\frac{1}{\sqrt{t}} + \frac{|n-t|^{3}}{t^{2}}\right) \right\}$.
	
	Let $E_{t}^{\upepsilon} = I_{t}^{\upepsilon} \cap 2\Z$ and $O_{t}^{\upepsilon} = I_{t}^{\upepsilon} \setminus E_{t}^{\upepsilon}.$ So $E_{t}^{\upepsilon}$ and $O_{t}^{\upepsilon}$ are the sets of even and odd integers in $I_{t}^{\upepsilon}$, respectively. Consider the sums 
	
	\begin{align*}
		S_{I} =& \sum_{n \in I_{t}^{\upepsilon}}\frac{e^{-\frac{(n-t)^{2}}{2t}}}{\sqrt{2\uppi t}}\exp\left\{ O\left( \frac{|n-t|^{3}}{t^{2}} \right)\right\}\\
		S_{E} =&\sum_{n \in E_{t}^{\upepsilon}}\frac{e^{-\frac{(n-t)^{2}}{2t}}}{\sqrt{2\uppi t}}\exp\left\{ O\left( \frac{|n-t|^{3}}{t^{2}} \right)\right\}\\
		S_{O} =&\sum_{n \in O_{t}^{\upepsilon}}\frac{e^{-\frac{(n-t)^{2}}{2t}}}{\sqrt{2\uppi t}}\exp\left\{ O\left( \frac{|n-t|^{3}}{t^{2}} \right)\right\}
	\end{align*}
	
	The sum $\sum_{n \in I_{t}^{\upepsilon}}\prob(N_{t} = n)a_{n}$ can be expressed
	
	$$
		\frac{1}{(2\uppi t)^{d/2}}\exp\left\{ O\left(\frac{1}{\sqrt{t}} \right) \right\}\left(S_{I} + c(S_{E}-S_{O}) \right).
	$$
	
	Sums of the same form as $S_{I}$, $S_{E}$, and $S_{O}$ are approximated in \cite[proof of Theorem 2.5.6]{lawler2010random} by way of comparison with the Gaussian integral $\frac{1}{\sqrt{2\uppi}}\int_{-\infty}^{\infty}e^{-u^{2}/2}\,du = 1$. That proof can be applied with nearly no modification to approximate each of the sums $S_{I}$, $S_{E}$, and $S_{O}$, and so we simply cite the result we need: there exists an $\upepsilon > 0$ so that
	
	$$
		S_{E}, S_{O} = \frac{1}{2} + O\left(\frac{1}{\sqrt{t}} \right) \quad\text{ and } \quad S_{I} = 1 + O\left(\frac{1}{\sqrt{t}} \right)
	$$ 
	
	so that $S_{I} - c(S_{E} - S_{0}) = 1 + O\left(\frac{1}{\sqrt{t}} \right)$. Then, using $\exp\left\{O\left(\frac{1}{\sqrt{t}} \right) \right\} = 1 + O\left(\frac{1}{\sqrt{t}} \right)$, it follows that
	
	$$
		\sum_{n \in I_{t}^{\upepsilon}}\prob(N_{t} = n)a_{n} = \frac{1}{(2\uppi t)^{d/2}}\left(1 + O\left(\frac{1}{\sqrt{t}} \right) \right)
	$$
	
	as desired. This completes the proof for $d \neq 2$.
	
	We finish by treating the $d = 2$ case. As mentioned before, the only difference is with the order of the error term. In the preceding proof, this becomes relevant in the approximation (\ref{a_asymptotic}). When the error factor for $a_{n}$ is $\exp\left\{ O\left( \frac{1}{n} \right) \right\}$, we argue that $\frac{1}{n} = O\left(\frac{1}{t}\right)$, uniformly for $n \in I_{t}^{\upepsilon}$, and so we can replace $\exp\left\{O\left( \frac{1}{n} \right) \right\}$ with $\exp\left\{O\left( \frac{1}{t} \right) \right\}$. This error turns out to be small enough to be disregarded, since, for instance, it is dominated by $O\left(\frac{|n-t|}{t} \right).$
	
	When $a_{n}$ has $\exp\left\{O\left(\frac{1}{\log^{2}n}\right)\right\}$ as an error factor, we still get $\frac{1}{\log^{2}n} = O\left(\frac{1}{\log^{2}t}\right)$ uniformly for $n \in I_{t}^{\upepsilon}$. However, this term cannot be disregarded and carries through until the end of the proof. In all, it contributes a $\left(1 + O\left(\frac{1}{\log^{2}t} \right) \right)$ error factor to the approximation. 

\end{proof}

\begin{proof}[Proof of Theorem \ref{two_norm_theorem} (iii)]

Using the decomposition (\ref{noisy_heat_kernel}), we have

\begin{equation}\label{triangle_inequality}
	\begin{aligned}
		|\upeta_{t} - \E^{\upeta}[\upeta_{t}]| =& \sum_{y \in \Z^{d}}|\upeta_{t}^{y} - \E[\upeta_{t}^{y}]|\\
		=&\sum_{y \in \Z^{d}}\bigg|\sum_{x \in \Z^{d}}\upeta^{x}(\upeta_{t}(x,y) - h_{t}(x,y)) \bigg|\\
		\leq&\sum_{y \in \Z^{d}}\sum_{x \in \Z^{d}}\upeta^{x}|\upeta_{t}(x,y) - h_{t}(x,y)|\\
		=&\sum_{x \in \Z^{d}}\upeta^{x}|\upeta_{t}(x,\cdot) - h_{t}(x,\cdot)|.
	\end{aligned}
\end{equation}

Let $r = r(t) \to \infty$ as $t \to \infty$, to be specified later. Fix $x \in \Z^{d}$ and let $A_{r}(x) = \{y\,:\,|x-y| \leq \sqrt{rt}\}$. We write
	
	\begin{align*}
		|\upeta_{t}(x, \cdot) - h_{t}(x, \cdot)| =& \sum_{y \in A_{r}(x)}|\upeta_{t}(x, y) - h_{t}(x, y)| + \sum_{y \not\in A_{r}(x)}|\upeta_{t}(x, y) - h_{t}(x, y)|\\
		\leq&\sum_{y \in A_{r}(x)}|\upeta_{t}(x, y) - h_{t}(x, y)| + \sum_{y \not\in A_{r}(x)}(\upeta_{t}(x, y) + h_{t}(x, y))\\
		\leq& \sqrt{\text{Vol}(A_{r}(x))\sum_{y \in A_{r}(x)}(\upeta_{t}(x,y) - h_{t}(x,y))^{2}} + \sum_{y \not\in A_{r}(x)}(\upeta_{t}(x, y) + h_{t}(x, y))\\
		\leq& \sqrt{\text{Vol}(A_{r}(x)) \norm{\upeta_{t}(x,\cdot) - h_{t}(x,\cdot)}^{2}} +\sum_{y \not\in A_{r}(x)}(\upeta_{t}(x, y) + h_{t}(x, y)).
	\end{align*}
	
	The third line results from an application of the Cauchy-Schwarz inequality. Now we apply expectations. Noting that $\text{Vol}(A_{r}(x)) = O((r t)^{d/2})$, we get
	
	$$
		\E\left[\sqrt{\text{Vol}(A_{r}(x)) \norm{\upeta_{t}(x,\cdot) - h_{t}(x,\cdot)}^{2}} \right] \leq \sqrt{ O((r t)^{d/2})\,\E[\,\norm{\upeta_{t}(x,\cdot) - h_{t}(x,\cdot)}^{2}\,]}
	$$
	
	using Jensen's inequality to move the expectation through the square root. Since by Theorem \ref{two_norm_theorem} we have that $\E[\norm{\upeta_{t}(x,\cdot) - h_{t}(x,\cdot)}^{2}]=o(t^{-d/2})$, it is clear that we can choose $r$ growing slowly enough so that the right-hand side above is $o(1)$ as $t \to \infty.$ Note that the choice of $r$ does not depend on $x$, as $\E[\,\norm{\upeta_{t}(x,\cdot) - h_{t}(x,\cdot)}^{2} \,]= o(t^{-d/2})$ uniformly in $x$.
	
	For the other term, we have
	
	$$
		\sum_{y \not\in A_{r}(x)}(\upeta_{t}(x, y) + h_{t}(x, y)) = 2\prob_{x}(|W_{t} - W_{0}| > \sqrt{r t})
	$$
	
	where $(W_{t})$ is the continuous-time simple random walk on $\Z^{d}$ with jump rate $1/2$. The probability on the right above is easily shown to be $o(1)$ for any $r \to \infty$ (for instance, using the same method as in the proof of Proposition \ref{spectral_radius}), and again uniformly in $x$. It follows that $\E[\,|\upeta_{t}(x,\cdot)-h_{t}(x,\cdot)|\,] = o(1)$ uniformly in $x$.
	
	Now, from (\ref{triangle_inequality}) and the dominated convergence theorem, 
	
	$$
		\E^{\upeta}\left[|\upeta_{t} - \E^{\upeta}[\upeta_{t}] \right] \leq \sum_{x \in \Z^{d}}\upeta^{x}\E\left[|\upeta_{t}(x, \cdot) - h_{t}(x, \cdot)|\right].
	$$
	
	Since each summand on the right is $o(1)$ uniformly in $x$ as $t \to \infty$, we may conclude that the entire sum is $o(1)$, completing the proof.
\end{proof}

Finally, we prove the central limit theorem for $(\upeta_{t})$.

\begin{proof}[Proof of Theorem \ref{clt}]
	Let $f: \R^{d} \to \R$ be continuous and bounded. We may assume by subtracting a constant that $\int_{\R^{d}}f\,d\upnu = 0$. Let $M = \sup_{\R^{d}}|f(x)|$. For ease of notation we write $f_{t}(y) = f(y/\sqrt{t/2}).$ We have
	
	$$
		\bigg| \sum_{y \in \Z^{d}}f_{t}(y) \upeta_{t}^{y}\bigg| \leq M\cdot|\upeta_{t} - \E^{\upeta}[\upeta_{t}]| + \bigg| \sum_{y\in \Z^{d}}f_{t}(y)\E^{\upeta}[\upeta_{t}^{y}]\bigg|.
	$$
	
	The first term on the right converges in probability to $0$ by Theorem \ref{two_norm_theorem} (iii). The second term can be expanded further as
	
	\begin{eqnarray}
		\bigg| \sum_{y \in \Z^{d}}f_{t}(y) \E^{\upeta}[\upeta_{t}^{y}]\bigg| &=& \bigg| \sum_{y \in \Z^{d}}\sum_{x \in \Z^{d}}f_{t}(y)\upeta^{x}h_{t}(x,y) \bigg|\nonumber\\
		&\leq& \sum_{x \in \Z^{d}}\upeta^{x}\bigg|\sum_{y \in \Z^{d}}f_{t}(y)h_{t}(x,y)\bigg|\label{clt_bound}
	\end{eqnarray}
	
	Note (\ref{clt_bound}) is a convergent sum for all $t$, as $\big|\sum_{y \in \Z^{d}}f_{t}(y)h_{t}(x,y)\big| \leq M$ for all $x$. We recognize $\sum_{y \in \Z^{d}}f_{t}(y)h_{t}(x,y)$ as $\E_{x}[f_{t}(W_{t})]$ where, as usual, $(W_{t})$ is a continuous-time simple random walk on $\Z^{d}$ with jump rate $1/2$. This term converges to $0$ for any $x \in \Z^{d}$ by the central limit theorem. It follows that (\ref{clt_bound}) is $o(1)$ as $t \to \infty$, which finishes the proof.

\end{proof}

\section{Conclusion}\label{conclusion}
We've shown that the averaging process on $\Z^{d}$ exhibits sharp concentration around its expectation with general initial conditions, and have used this to show a central limit theorem for the process which extends results in \cite{nagahata2009central} and \cite{nagahata2010note}. We finish by giving some potential directions for future work and considering a short example.
\subsection{Future work}

Here we consider the possibility of applying the techniques of this paper to more general stochastic flows $\upeta_{t} = A_{t}\upeta_{t^{-}}$. The motivating question is: what conditions on the matrices $(A_{t})$ ensure concentration of the type given in Theorem \ref{two_norm_theorem} for the corresponding process?

Two reasonable generalizations are to allow for long-range interactions between vertices, and to allow for non-conservative flows. In terms of the matrices $(A_{t})$, the first modification means that the $A_{t}$'s may now differ from the identity $I$ on larger blocks than just $2 \times 2$ submatrices; the latter means that the $A_{t}$'s need not be stochastic. (Note that the matrices corresponding to the averaging process are stochastic and differ from the identity only on a $2 \times 2$ block.) One could also consider different local mass redistribution dynamics, as we do in the final subsection.

\subsection{Potlach process}\label{potlach}

Consider the potlach process, another random mass distribution process which was introduced in \cite{spitzer1981infinite}. In the simplest variant, each vertex in $\Z^{d}$ has an independent, rate-$1$ Poisson clock; upon a ring at $x$, vertex $x$ distributes all of its present mass evenly among its neighbors, keeping none for itself. This process was also studied under the name \textit{meteor process} in \cite{billey2015meteors} and \cite{burdzy2015meteor}, the idea being that meteors hit each vertex at rate $1$, blowing the mass onto neighboring vertices upon impact.

Suppose we begin the potlach process from a unit mass at some vertex. The expected mass distribution at time $t$ is again the distribution of a continuous-time random walk on $\Z^{d}$ started from the the initial vertex, this time with jump rate $1$. The dynamics of the two-point coupled random walks $(\wt{W}_{t}^{(1)}, \wt{W}_{t}^{(2)})$ for the potlach process can be derived using the same tools we used for the averaging process in Sections \ref{construction_duality} and \ref{discrete_rw}. One examines the difference process $(\wt{D}_{t}) = (\wt{W}_{t}^{(1)}-\wt{W}_{t}^{(2)})$ and then passes to the corresponding discrete time walk $(\wt{X}_{n})$. Transition probabilities for the walk $(\wt{X}_{n})$ for dimension $d = 1$ are given in Figure \ref{potlach_fig}.

\begin{figure}[h]
	\captionsetup{width=0.8\textwidth}
	\captionsetup{font=small}
	\centering
		\includegraphics[scale = 0.425]{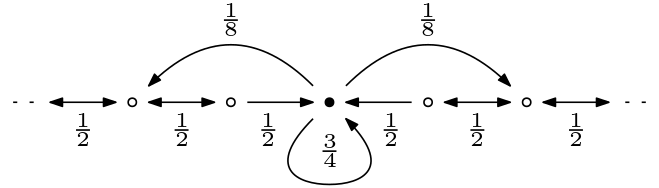}
		\caption{A representation of the discrete-time walk $(\wt{X}_{n})$ corresponding to the poltach process in dimension $d = 1$. The black vertex is the origin. The transition probabilities are those of a simple random walk except for a perturbation at the origin.}
		\label{potlach_fig}	
\end{figure}

Using the same proof method as Proposition \ref{gf_prop} (and in fact it is slightly easier in this case), one can derive the relation

$$
	\wt{G}(z) = \frac{2zG(z)}{1 - (1-z)^{2}G(z)}.
$$

Once again, the dominant singularities are $z = \pm 1$, and we can compute rough singular expansions at each of them. As $z \to 1$, $\wt{G}(z) \approx 2G(z)$, for any $d$. As $z \to -1$, $\wt{G}(z) \approx f(z) + o(G(z))$ for $d \leq 2$, and $\wt{G}(z) \approx f(z) + c_{d}G(z)$ for $d \geq 3$, where $f(z)$ is analytic at $z = -1$ and depends on $d$.

After translating back to continuous time as in the proof of Theorem \ref{two_norm_theorem}, these expansions imply that as $t \to \infty$, $\prob(\wt{D}_{t} = \bm{0}) = (2+o(1))\prob(D_{t} = \bm{0})$, and hence $\E[ \norm{\upeta_{t} - \E[\upeta_{t}]}^{2}] = \Uptheta(t^{-d/2})$. This is a strictly larger order of concentration than is exhibited by the averaging process; in particular, we cannot deduce the central limit theorem for the potlach process from the concentration alone as we did in the proof of Theorem \ref{clt}. Of course, this is not to say that the central limit theorem is false for the potlach process---indeed, it holds for $d \geq 3$ by the results in (\cite{nagahata2009central}, \cite{nagahata2010note}). 

\subsection{Acknowledgements}

We thank Xavier P\'{e}rez Gim\'{e}nez and Juanjo Ru\'{e} for useful discussions. 

\newpage

\bibliographystyle{alpha}
\bibliography{refs} 

\end{document}